\documentclass[bjps,preprint]{imsart}

\RequirePackage[OT1]{fontenc}
\usepackage{amsthm,amsmath,natbib}
\RequirePackage{hyperref}

\usepackage[english]{babel}
\usepackage{color}

\usepackage{amsmath, latexsym, amsthm, amsfonts,amssymb}
\usepackage[text={5.8in,8.4in},centering]{geometry}
\usepackage[dvipsnames]{xcolor}
\usepackage{algorithm}
\usepackage{algpseudocode}

\usepackage{epsfig}

\usepackage{extarrows}
\usepackage{graphicx}
\usepackage{float}
\usepackage{multirow}
\usepackage{caption}
\usepackage{subfigure}
\usepackage{booktabs}

\newcommand{\dd}{{\,\mathrm d}}

\renewcommand{\th}{\theta}

\newcommand{\ga}{\gamma}

\newcommand{\eps}{\varepsilon}

\renewcommand{\phi}{\varphi}

\newcommand{\scr}[1]{{\mathcal #1}}

\newcommand{\EE}{\mathbb{E}}
\newcommand{\FF}{\mathbb{F}}

\newcommand{\PP}{\mathbb{P}}
\newcommand{\ind}{\mathbf{1}}

\newcommand{\Bm}{\begin{bmatrix}}
\newcommand{\Em}{\end{bmatrix}}

\newcommand{\Th}{\Theta}

\newcommand{\simind}{\stackrel{\rm ind}{\sim}}
\newcommand{\simiid}{\stackrel{\rm iid}{\sim}}

\newcommand{\cit}{\cite}

\startlocaldefs
\numberwithin{equation}{section}
\theoremstyle{plain}
\newtheorem{thm}{Theorem}[section]

\newtheorem{lem}[thm]{Lemma}
\newtheorem{cor}[thm]{Corollary}
\newtheorem{rem}[thm]{Remark}

\newtheorem{ass}[thm]{Assumption}

\endlocaldefs

\begin{document}
	
\begin{frontmatter}
		
% "Title of the Paper"
\title{Bayesian estimation of a decreasing density}
		
\runtitle{Bayesian estimation of a decreasing density}

\begin{aug}
\author{\fnms{Geurt} \snm{Jongbloed}\thanksref{a}\ead[label=e1]{g.jongbloed@tudelft.nl}},
\author{\fnms{Frank} \snm{van der Meulen}\thanksref{a}\ead[label=e2]{f.h.vandermeulen@tudelft.nl}}
\and
\author{\fnms{Lixue} \snm{Pang}\thanksref{a}%
\ead[label=e3]{l.pang@tudelft.nl}}
					
\runauthor{G. Jongbloed et al.}
					
\affiliation[a]{Institute of Applied Mathematics,
Delft University of Technology}					
				
\address{Address of the First, Second and Third authors\\
Van Mourik Broekmanweg 6, 2628 XE Delft, The Netherlands.\\
\printead{e1,e2,e3}}

\end{aug}

\begin{abstract}
Suppose $X_1,\ldots, X_n$ is a random sample from  a bounded and decreasing density $f_0$ on $[0,\infty)$. We are interested in estimating such $f_0$, with special interest in $f_0(0)$. This problem is encountered in various statistical applications and has gained quite some attention in the statistical literature.  It is well known that the  maximum likelihood estimator is inconsistent at zero. This has led several authors to propose   alternative estimators which are consistent.
As any decreasing density can be represented as a scale mixture of uniform densities,  a Bayesian estimator is  obtained by endowing the mixture distribution with the Dirichlet process prior. Assuming this prior, we derive contraction rates of the posterior density at zero by carefully revising arguments presented in  \cit{Salomond}. Several choices of base measure are numerically evaluated and compared.  In a simulation  various frequentist  methods and a Bayesian estimator are compared. Finally,  the Bayesian procedure is applied to current durations data described in \cit{Keiding}.
\end{abstract}

\begin{keyword}[class=MSC]
	\kwd[Primary ]{62G07}
	\kwd[; secondary ]{62N01}
\end{keyword}
		
\begin{keyword}
\kwd{Bayesian nonparametrics}
\kwd{posterior consistency}
\kwd{contraction rate}
\end{keyword}

\end{frontmatter}

 %%%%%%%%%%%%%
\section{Introduction}
\subsection{Setting}
Consider an  independent and identically distributed  sample $X_1,\dots,X_n$  from a bounded decreasing density $f_0$ on $[0,\infty)$. The problem of estimating $f_0$ based on the sample, only using the information that it is decreasing, has attracted quite some attention in the literature.
One of the reasons for this is that the estimation problem arises naturally in several applications. 

To set the stage, we discuss a simple idealized example related to the waiting time paradox. Suppose buses arrive at a bus stop at random times, with independent interarrival times sampled from a distribution with distribution function $H_0$.  At some randomly selected time, somebody arrives and has to wait for a certain amount of time until the next bus arrives. A natural question then is: `what is the distribution of the remaining waiting time until the next bus arrives?' 
In order to derive this distribution, two observations are important. 

The first is, that the time of arrival of the traveller is more likely contained in a long interarrival interval than a short interarrival interval. Under mild assumptions, one can show that actually the length of the whole interarrival interval (so between arrival of the previous and the next bus) containing the time the traveller arrives, can be viewed as a draw from the length biased distribution associated to distribution function $H_0$. This is the distribution with distribution function
\begin{equation}
\label{eq:lengthbiased}
\bar{H}_0(y)=\frac1{\mu_{H_0}}\int_0^y z\,dH_0(z) \mbox{ with }\mu_{H_0}=\int_0^\infty z\,dH_0(z).
\end{equation}
It is assumed that $0<\mu_{H_0}<\infty$.

The second observation is that the remaining waiting time for the traveller is a uniformly distributed fraction of the interarrival time. A residual waiting time $X$ is therefore interpreted as
$$
X=U Y,
$$
where $U$ is uniformly distributed on $(0,1)$ and, independently of $U$,  $Y$ according to distribution function $\bar{H}_0$ defined in (\ref{eq:lengthbiased}). 

These observations imply that on $[0,\infty)$, $X$ has survival function
\begin{eqnarray*}
	P(X>x)&=&P(UY>x)=\int_{y=x}^{\infty}\int_{u=x/y}^1du\,d\bar{H}_0(y)=\int_{y=x}^\infty\left(1-\frac{x}{y}\right)\,d\bar{H}_0(y)\\
	&=&\frac1{\mu_{H_0}}\int_{y=x}^\infty\left(y-x\right)\,dH_0(y)=\frac1{\mu_{H_0}}\int_{y=x}^\infty\left(1-H_0(y)\right)\,dy, 
\end{eqnarray*}
using integration by parts in the last step. Differentiating with respect to $x$, yields the following relation between the sampling density $f_0$ and distribution function $H_0$:  
\begin{equation}
\label{eq:expre}
f_0(x)=\frac1{\mu_{H_0}}\left(1-H_0(x)\right), \,\, x\ge0.
\end{equation}
In words: the sampling density is proportional to a survival function of the interarrival distribution, which is by definition decreasing. Note that in the classical waiting time paradox, the underlying arrival process is taken to be a homogeneous Poisson process, with exponential interarrival times. In view of (\ref{eq:expre}), this leads to the `paradox' that the distribution of the residual waiting time equals the distribution of the interarrival time itself.

More examples where exactly this model comes into play can for instance be found in the introductory section of \cit{Kulikov}, in \cit{Vardi}, \cit{Watson}, \cit{Keiding} and references therein. In those examples, the challenge is to estimate the interarrival distribution function $H_0$ based on a sample from density $f_0$. To do this, the `inverse relation' of (\ref{eq:expre}), expressing $H_0$ in terms of $f_0$ can be employed:
\begin{equation}
\label{eq:invrel}
H_0(x)=1-\mu_{H_0}f_0(x)=1-\frac{f_0(x)}{f_0(0)},\,\,x\ge0.
\end{equation}
Here it is used that $H_0(0)=0$.

From (\ref{eq:invrel}) it is clear that in order to estimate $H_0$ at some specific point $x>0$, estimating the decreasing sampling density $f_0$ at zero is of special interest. This value  occurs at the right hand side for any choice of $x>0$.

\subsection{Literature overview}
The most commonly used estimator for $f_0$ is the maximum likelihood estimator derived in \cit{Grenander}.
This estimator is defined as the maximizer of the log likelihood $\ell(f)=\sum_{i=1}^n\log f(X_i)$
over all decreasing density functions on $(0,\infty)$. The solution $\hat{f}_n$ of this maximization problem can be graphically constructed. Starting from the empirical distribution $\FF_n$ based on $X_1,\ldots,X_n$, the least concave majorant of $\FF_n$ can be constructed. This is a concave distribution function. The left-continuous derivative of this piecewise linear concave function yields the maximum likelihood (or Grenander) estimator for $f_0$. For more details on the derivation of this estimate, see  Section 2.2 in \cit{GroJoCUP}.
As can immediately be inferred from the characterization of the Grenander estimator,
$$
\hat{f}_n(0):=\lim_{x\downarrow0}\hat{f}_n(x)=\max_{1\le i\le n}\frac{\FF_n(X_i)}{X_i}\ge \frac{\FF_n(X_{(1)})}{X_{(1)}}=\frac1{nX_{(1)}},
$$
where $X_{(i)}$ denotes the $i$-th order statistic of the sample.  Denoting convergence in distribution by $\stackrel{d}{\rightarrow}$,
$$
nf_0(0)X_{(1)} \stackrel{d}{\rightarrow} Y \quad \mbox{as} \quad n\rightarrow\infty
$$
where $Y$ has the standard exponential distribution. It is clear that $\hat{f}_n(0)$ does not converge in probability to $f_0(0)$.  This inconsistency of $\hat{f}_n(0)$ was first studied in \cit{WoodroofeSun}. There it is also shown that
$$
\frac{\hat{f}_n(0)}{f_0(0)}\stackrel{d}{\rightarrow} \sup_{t>0}\frac{N(t)}{t}\stackrel{d}{=}\frac1U \quad \mbox{as} \quad n\rightarrow\infty,
$$
where $N$ is a standard Poisson process on $[0,\infty)$ and $U$ is a standard uniform random variable.

It is clear from (\ref{eq:invrel}) that this inconsistency is undesirable, as estimating the distribution function of interest, $H_0$, at any point $x>0$, requires estimation of $f_0(0)$. Various approaches have been taken to obtain a consistent estimator of $f_0(0)$. The idea in \cit{Kulikov} is to estimate $f_0(0)$ by  $\hat{f}_n$ evaluated at a small positive (but vanishing) number: $\hat{f}_n(cn^{-1/3})$ for some  $c>0$. There it is shown that  the estimator is $n^{1/3}$-consistent, assuming $f_0(0)<\infty$ and $|f_0^\prime(0)|<\infty$.

A likelihood related approach was taken in \cit{WoodroofeSun}. There a penalized log likelihood function is introduced, where the estimator is defined as maximizer of
$$
\ell_{\alpha}(f)=\sum_{i=1}^n\log f(X_i)-\alpha n f(0).
$$
For fixed $\alpha\ge0$, this estimator can be computed explicitly by first transforming the data using a data dependent affine transformation and then applying the basic concave majorant algorithm to the empirical distribution function based these transformations data. It is shown (again, assuming $f_0(0)<\infty$ and $|f_0^\prime(0)|<\infty$) that the optimal rate to choose $\alpha$ is $n^{-2/3}$. Then, the maximum penalized estimator $\hat{f}_{n,\hat{\alpha}_n}^P(0)$ is $n^{1/3}$-consistent.

\cit{GroJoCUP} proposed to estimate $f_0(0)$ by the histogram estimator  $b_n^{-1}\mathbb{F}_n(b_n)$, where $\{b_n\}$ is a sequence of positive numbers with $b_n\to 0$ if $n\to \infty$. The bin widths $b_n$ can e.g. be chosen by estimating the asymptotically Mean Squared Error-optimal choice. Also this estimator is $n^{1/3}$- consistent assuming $f_0(0)<\infty$ and $\mid f'_0(0)\mid<\infty$.

\subsection{Approach}
In this paper we take a Bayesian nonparametric approach to the problem. An advantage of the Bayesian setup is the ease of constructing credible regions. To construct frequentist analogues of these, confidence regions, can be quite cumbersome, relying on either bootstrap simulations or asymptotic arguments.

 To formulate a Bayesian approach for estimating a decreasing density, note that any decreasing density on $[0,\infty)$ can be represented as a scale mixture of uniform densities (see e.g. \cit{Wils}):
\begin{equation}f_G(x)=\int_0^\infty \psi_x(\th) dG(\th),\mbox{ where } \psi_x(\theta)=\theta^{-1}1_{[0,\th]}(x),\label{eq:scalemix}\end{equation}
where $G$ is a distribution function concentrated on the positive half line.
Therefore, by endowing the mixing measure with a prior distribution we obtain the posterior distribution of the decreasing density, and in particular of $f_0(0)$. A convenient and well studied prior for distribution functions on the real line is the  Dirichlet process (DP) prior (see for instance \cit{ferguson} and \cit{vdV-Ghosal-book}). This prior contains two parameters: the concentration parameter, usually denoted by $\alpha$, and the base probability distribution, which we will denote by  $G_0$.
The approach where a prior is obtained by putting a Dirichlet process prior on $G$ in \eqref{eq:scalemix} was previously considered in \cit{Salomond}. In that paper, the asymptotic properties of the posterior in a frequentist setup are studied.  More specifically, contraction rates are derived to quantify the performance of the Bayesian procedure. This is a rate for which we can shrink balls around the true parameter value, while maintaining most of the posterior mass. More formally, if
  $L$ is  a semimetric on the space of density functions,  a contraction rate $\eps_n$  is a sequence of positive numbers  $\eps_n\downarrow 0$ for which the posterior mass of the  balls $\{f\,:\, L(f,f_0)\le \eps_n\}$ converges in probability to $1$ as $n\to \infty$, when assuming $X_1,X_2,\ldots$ are independent and identically distributed  with density $f_0$. A general discussion on contraction rates is given in Chapter 8 of \cit{vdV-Ghosal-book}.

\subsection{Contributions}
In Theorem 4 in \cite{Salomond} the rate $(\log n / n)^{2/9}$ is derived for pointwise loss at any $x>0$. For $x=0$, only posterior consistency is derived, essentially under the assumption that the base measure admits a density $g_0$ for which there exists $1<a_1\le a_2$ such that $ e^{-a_1/\th} \lesssim g_0(\th) \lesssim e^{-a_2/\th}$ when $\th$ is sufficiently small (theorem 4).  These are interesting results, though one would hope to prove the rate $n^{-1/3}$ for all $x\ge 0$. Under specific conditions on the underlying density, this rate is attained by estimators to be discussed in section \ref{sec:methods}. We explain why the  techniques in the proof of \cite{Salomond} cannot be used to obtain rates at zero and present an alternative proof (using different arguments). This proof not only reveals consistency, but also yields a contraction rate  equal to $n^{-2/9}$ (up to log factors) that coincides with the  case $x>0$. 
We argue that with the present method of proof a better rate is not easily obtained.
Many  results from \cit{Salomond} are important ingredients to the proof we present. The first key contribution of this paper is to derive the claimed contraction rate, combining some of Salomond's results with new arguments.

We also address computational aspects of the problem and show how draws from the posterior can be obtained using the algorithm presented in \cit{Neal}. Using this algorithm we conduct four studies.
\begin{itemize}
  \item For a fixed dataset, we compare the performance of the posterior mean under various choices of base measure for the Dirichlet process.
  \item  We investigate empirically the rate of convergence of the Bayesian procedure for estimating the density at zero when $g_0(\th) \sim e^{-1/\th}$ or $g_0(\th) \sim \th$ for $\th \downarrow 0$. The simulation results  suggest that for both choices of base measure the rate is $n^{-1/3}$. If   $g_0(\th) \sim e^{-1/\th}$ this implies that the derived rate $n^{-2/9}$ (up to log factors) is indeed  suboptimal, as anticipated by \cite{Salomond}.  If $g_0(\th) \sim \th$ the rate $n^{-1/3}$ is interesting, as it contradicts the belief   that ``due to the similarity to the maximum likelihood estimator, the posterior distribution is in this case not consistent`` (page 1386 in \cite{Salomond}).
\item  We  compare the behaviour of various proposed frequentist methods and the Bayesian method for estimating $f_0(0)$. Here we vary the sample sizes and consider both the Exponential and half-Normal distribution as true data generating distributions.
\item Pointwise credible sets can be approximated in a direct way from MCMC-output, which is much more straightforward than the construction of frequentist confidence intervals based on  large-sample limiting results. 

\end{itemize}

\subsection{Outline}
In section \ref{sec:consistency} we  derive pointwise contraction rates for the density evaluated at $x$, for any $x\ge 0$. In section \ref{sec:alg}  a Markov Chain Monte Carlo method for obtaining draws from the posterior is given, based on the results of  \cit{Neal}. This is followed by a review of some existing methods to consistently estimate $f_0$ at zero. Section \ref{sec:simul} contains numerical illustrations. The appendix contains some technical results.

\subsection{Frequently used notation}
For two sequences $\{a_n\}$ and $\{b_n\}$ of positive real numbers, the notation $a_n\lesssim b_n$  (or $b_n\gtrsim a_n$) means that there exists a constant $C>0$ that is independent of $n$ and such that $a_n\leq C b_n.$ We write $a_n\asymp b_n$ if both $a_n\lesssim b_n$ and $a_n\gtrsim b_n$ hold.  We denote by  $F$ and $F_0$ the cumulative distribution functions corresponding to the probability densities $f$ and $f_0$ respectively.  We denote the $L_1$-distance between two density functions $f$ and $g$ by $L_1(f,g)$, i.e.\ $L_1(f,g) = \int |f(x)-g(x)| \dd x$. The Kullback-Leibler divergence `from $f$ to $f_0$' is denoted by $KL(f,f_0)=\int f(x)\log\frac{f(x)}{f_0(x)}dx$.

%%%%%%%%%%%%%%%%
\section{Pointwise posterior contraction rates}
\label{sec:consistency}
Let $\scr{F}$ denote the collection of all bounded decreasing densities on $[0,\infty)$ and recall that $X_1,X_2,\dots$ are i.i.d. with density $f\in\mathcal{F}$. Denote the distribution of $X^n=(X_1,\ldots, X_n)$ under $f$ by $\PP_f$ and expectation under $\PP_f$ by $\EE_f$. In this section we are interested in the asymptotic behaviour of the posterior distribution of $f(x)$  in a frequentist setup. This entails that we study the behaviour of the posterior distribution on $\scr{F}$ while assuming a true underlying density $f_0$. Set $\PP_0=\PP_{f_0}$ and $\EE_0=\EE_{f_0}$. Denote the prior measure on $\mathcal{F}$ by $\Pi$ and the posterior measure by $\Pi(\cdot \mid X^n)$.

Given a loss function $L$ on $\scr{F}$, we say that the posterior is consistent with respect to $L$ if for any $\eps>0$, $\EE_0 \Pi(L(f,f_0) >\eps \mid X^n) \to 0$ when $n\to \infty$. If $\{\eps_n\}$ is a sequence that tends to zero, then we say that the posterior contracts  at rate $\eps_n$ (with respect to $L$) if $\EE_0 \Pi(L(f,f_0) >\eps_n \mid X^n) \to 0$ when $n\to \infty$. The rate $\{\eps_n\}$ is called a contraction rate.

\cit{Salomond} derived contraction rates based on the Dirichlet process prior for the $L^1-$, Hellinger- and pointwise loss function.

In the following theorem we  derive sufficient conditions  for posterior contraction in terms of the behaviour of the density of the  base measure near zero. In that, we closely follow the line of proof in \cit{Salomond}.
Although the argument in \cit{Salomond} for proving posterior contraction rate $\epsilon_n$ for $f_0(x)$ with $x>0$ is correct, we prove the theorem below for $x\ge0$ rather than only for $x=0$. The reason for this is twofold: {\it (i)} many steps in the proof for $x>0$ are also used in the proof for $x=0$; {\it (ii)} we obtain one theorem covering pointwise contraction rates for all $x\ge 0$.  For the base measure we have the following assumption.

\begin{ass}\label{ass:base}
The base distribution function of prior, $G_0$, has a strictly positive Lebesgue density $g_0$ on $(0,\infty)$. There exists positive numbers  $\th_0, \underline{a},  \underline{k}, \overline{a}$ such that
\begin{equation}\label{eq:priormass}
\underline{k} e^{-\underline{a}/\th}\le g_0(\theta)\le \theta^{\overline{a}}\quad \mbox{for all} \quad \th \in (0,\theta_0).
\end{equation}	
\end{ass}

For the data generating density we assume
\begin{ass}\label{ass:f_0}
The data generating density   $f_0 \in \scr{F}$ and
\begin{itemize}
  \item there exists an $x_0>0$ such that $\sup_{x\in[0,x_0]} |f_0^{'}(x)|<\infty$;
  \item the exist positive constants $\beta$ and $\tau$ such that $f_0(x) \le e^{-\beta x^\tau}$ for $x$ sufficiently large.
\end{itemize}
\end{ass}
 Theorem 2 in \cite{Salomond} asserts the existence of a positive constant $C$ such that
\[ \Pi\left(f \in \scr{F} \colon  L_1(f, f_0) \ge C \left(\frac{\log{n}}{n}\right)^{1/3} (\log n)^{1/\tau} \mid X^n \right) \to 0, \]
$ \PP_0-\text{almost surely} (n\to \infty)$. This result will be used in the proof for deriving an upper bound on the pointwise contraction rate of the posterior at zero.

Define a sequence of subsets of $\mathcal{F}$ by \[ \scr{F}_n=\{f\in\mathcal{F}\,:\, f(0)-f(x)\le M_n x, \mbox{ for all}\,x\in[0,\xi_n]\},\] where $\xi_n\asymp n^{-2/9}$ and $M_n\asymp (\log n)^{\beta}$.
\begin{thm}\label{th:centralth}  Let $X_1, X_2,\ldots$ be independent random variables, each with density $f_0$ satisfying assumption \ref{ass:f_0}. Let  $\Pi_n$ be the prior distribution on $\scr{F}_n$ that is obtained via (\ref{eq:scalemix}), where  $G \sim DP(G_0, \alpha)$ and $G_0$ satisfies assumption \ref{ass:base}. Assume $\beta>1/3$ (in the behaviour of the sequence $\{M_n\}$). For any   $x\in [0,\infty)$ with $f_0^{'}(x)<0$ there exists a constant $C>0$ such that,
\[ 	\EE_0\Pi\left(f\in\mathcal{F}_n \colon  |f(x)-f_0(x)| > C n^{-2/9}(\log n)^\beta\,\Big|\, X^n\right)\to 0.  \]
for  $n\to\infty$.
\end{thm}

In the proof we will use the following  lemma (see appendix B and lemma 8 of \cite{Salomond}).
\begin{lem}\label{lem:KLn}
Let $\epsilon_n=(\log n/n)^{1/3}$ and $f_0$ satisfy assumption \ref{ass:f_0}.
Define
\begin{equation}\label{eq:defDn}D_n=\int\prod_{i=1}^n\frac{f(X_i)}{f_0(X_i)}d\Pi(f).\end{equation} There exist strictly positive constants $c_1$ and $c_2$  such that
\begin{equation}
\label{denominator}
\PP_0\left(D_n< c_1 e^{-c_2 n\epsilon_n^2}\right)=o(1) \quad \mbox{as} \quad n\to \infty.
\end{equation}
\end{lem}

We now give the proof of Theorem \ref{th:centralth}.
\begin{proof}[Proof of Theorem \ref{th:centralth}]
The posterior measure of a measurable set $\scr{E}\subset\scr{F}$ is given by
\[\Pi(\scr{E} \mid X^n)=D_n^{-1}\int_{\scr{E}}\prod_{i=1}^n\frac{f(X_i)}{f_0(X_i)}d\Pi(f),
\]
where $D_n$ is as defined in (\ref{eq:defDn}). By lemma \ref{lem:KLn} there
exist positive constants $c_1$ and $c_2$ such that $\PP_0(\scr{D}_n) = o(1)$, where  $\scr{D}_n=\{D_n< c_1 e^{-c_2 n\epsilon_n^2}\}$.
Let $C>0$. Define $\eta_n=n^{-2/9}(\log n)^\beta$, $B_n(x)=\{f\in\mathcal{F}_n \colon  \mid f(x)-f_0(x)\mid>C\eta_n\}$ and consider (test-) functions $\Phi_{n} : \mathbb{R} \rightarrow [0,1]$.
We bound
 \begin{align}\label{eq:bound-Bn2}
\EE_0\Pi(&B_{n}(x) \mid X^n) \nonumber \\
&=\EE_0\Pi(B_{n}(x)\mid X^n)1_{\scr{D}_n}+\EE_0\Pi(B_{n}(x)\mid X^n)1_{\scr{D}^c_n}\Phi_{n}(x)\nonumber\\
&\qquad \qquad\qquad\qquad \qquad\qquad \qquad  +\EE_0\Pi(B_{n}(x)\mid X^n)1_{\scr{D}^c_n}(1-\Phi_{n}(x))\nonumber  \\
&\le \EE_0\left[1_{\scr{D}_n}\right]+\EE_0(\Phi_{n}(x))+\EE_0\left[ D_n^{-1}\int_{B_{n}(x)}\prod_{i=1}^n\frac{f(X_i)}{f_0(X_i)}(1-\Phi_{n}(x))d\Pi(f)1_{\scr{D}^c_n}\right] \nonumber \\
&\le  \PP_0\left(\scr{D}_n\right)+\EE_0(\Phi_{n}(x))+c_1^{-1} e^{c_2 n\epsilon_n^2}\EE_0 \int_{B_{n}(x)}\prod_{i=1}^n\frac{f(X_i)}{f_0(X_i)}(1-\Phi_{n}(x))d\Pi(f)\nonumber  \\
&= o(1)+\EE_0(\Phi_{n}(x))+c_1^{-1}e^{c_2 n\epsilon_n^2}\int_{B_{n}(x)}\EE_f(1-\Phi_{n}(x))d\Pi(f).
\end{align}

To construct the specific test functions $\Phi_{n}(x)$, we distinguish between $x>0$ and $x=0$. For case $x>0$,
it follows from the  proofs of theorems 3 and 5 in \cit{Salomond} that there exists a sequence test functions such that
\begin{align*}
\EE_0\, \Phi_{n}(x) &=o(1)\\
\sup_{f \in B_{n}(x)}\EE_f (1-\Phi_{n}(x)) &\le e^{-C'n(C\eta_n)^3}= e^{-C'C^3n\epsilon_n^2}.
\end{align*}
for some constant $C'>0$. Substituting these bounds into \eqref{eq:bound-Bn2} and choosing $C>(c_2/C')^{1/3}$ shows that $\EE_0\Pi(B_{n}(x)  \mid X^n)\to 0$ as $n\to \infty$.
This finishes the proof for $x>0$.

We now consider the case $x=0$. 
Define subsets
\begin{align*}
B_{n}^+(0)& =\{f\in\mathcal{F}_n\colon  f(0)-f_0(0)>C\eta_n\}\\
B_{n}^-(0)& =\{f\in\mathcal{F}_n\colon f(0)-f_0(0)<-C\eta_n\}.
\end{align*}
As $B_{n}(0)=B_{n}^+(0)\cup B_{n}^-(0)$, $\Pi(B_{n}(0)\mid X^n) \le \Pi(B_{n}^+(0)\mid X^n) + \Pi(B_{n}^-(0)\mid X^n)$.
For bounding $\EE_0\Pi( B_{n}^-(0)\mid X^n)$, use the same test function defined in \cit{Salomond}. %$\Phi_{n}^-(0)$ by  $\Phi_{n}^-(0):=\phi_{n}^{-}(0)$, with $\phi_n^{-}(0)$ as defined in (\ref{eq:phiminus}).
Then it follows from the inequalities in \eqref{eq:bound-Bn2}, applied with $B_{n}^{-}(0)$ instead of $B_{n}(x)$, that
$\EE_0\Pi( B_{n}^-(0)\mid X^n)=o(1)$ as $n\to\infty$.

For bounding $\EE_0 \Pi(B_{n}^+(0)\mid X^n)$, we also use the inequalities in \eqref{eq:bound-Bn2}, applied with $B_{n}^{+}(0)$ instead of $B_{n}(x)$. However, we also intersect with the event
 \[ A_n = \{f\colon L_1(f,f_0) \le C \eps_n (\log n)^{1/\tau}\} \]
to obtain
\[ \EE_0\Pi(B_{n}^+(0) \mid X^n)
\le  o(1)+\EE_0(\Phi_{n}(0))+c_1^{-1} e^{c_2 n\epsilon_n^2}\int_{B_{n}^+(0)\cap A_n}\EE_f(1-\Phi_{n}(0))d\Pi(f).\]
This holds true since theorem 2 in \cite{Salomond} gives $\Pi(A_n^c \mid X^n) \to 0$, $\PP_0$-almost surely.

Now define
$$\Phi_{n}^+(0)=1\left\{n^{-1}\sum_{i=1}^n 1_{[0,\xi_n]}(X_i)-\int_0^{\xi_n}f_0(t)dt>\tilde{c}_n\right\},$$
where
\begin{equation}\label{eq:xi_n-tildec_n} \xi_n\asymp n^{-2/9} \qquad \text{and} \qquad \tilde{c}_n = C\xi_n\eta_n/3\asymp n^{-4/9}(\log n)^\beta.\end{equation}
%Note that   $\tilde{c}_n= C\xi_n\tilde{\eta}_n/3$.
By Bernstein's inequality (\cit{AWVaart}, lemma 19.32),

  $$\EE_0\, \Phi_{n}^+(0) \le 2\exp\left(-\frac14 \frac{n\tilde{c}_n^2}{M\xi_n+\tilde{c}_n}\right)
  =o(1).$$
Here we bound the second moment of $1_{[0,\xi_n]}(X_i)$ under $\PP_0$ by $f_0(0)\xi_n$  and use that   $f_0(0)\le M$.

It remains to bound
\[ I:=e^{c_2 n\epsilon_n^2}\int_{B_{n2}^+(0)\cap A_n}\EE_f(1-\Phi^+_{n}(0))d\Pi(f). \]
Since both $f$ and $f_0$ are nonincreasing we have
\[ \int_0^{\xi_n} (f(t)-f_0(t)) \dd t \ge (f(\xi_n)-f_0(0))\xi_n. \]
Hence
\begin{align*}
	\int_0^{\xi_n} f_0(t) \dd t &\le \int_0^{\xi_n} f(t) \dd t +(f_0(0)-f(\xi_n)) \xi_n \\ & \le \int_0^{\xi_n} f(t) \dd t + \xi_n (f_0(0) -f(0)+M_n\xi_n),
\end{align*}
the final inequality being a consequence of $f\in \mathcal{F}_n$.
 Since for $f\in B_{n}^+(0)$ we have $f_0(0)-f(0)\le -C\eta_n$ we get
\[ \int_0^{\xi_n} f_0(t) \le \int_0^{\xi_n} f(t) \dd t +\xi_n (M_n\xi_n - C\eta_n). \]
Using the derived bound we see that
\[ I_2 \le e^{c_2 n\epsilon_n^2}  \int_{B_{n}^+(0) \cap A_n} \PP_f\left(\sqrt{n}\left(\frac1{n} \sum_{i=1}^n \ind_{[0,\xi_n]}(X_i)-\int_0^{\xi_n} f(t) \dd t\right) \le - v_n\right) \dd \Pi(f),\]
where
\begin{equation}\label{eq:v_n} v_n = -\sqrt{n} \left(\tilde{c}_n + \xi_n (M_n\xi_n - C\eta_n)\right). \end{equation}
Note that $M_n\xi_n\asymp\eta_n$, by choice of $M_n, \xi_n$. Taking $C$ big enough such that $M_n\xi_n \le C\eta_n/3$ we have
 $v_n \ge C \sqrt{n} \eta_n \xi_n/3$
 is positive (recall that $\tilde{c}_n$ is defined in \eqref{eq:xi_n-tildec_n}). Using that $f$ is nonincreasing and that $f\in A_n$ we get
\begin{align*} \EE_f \ind_{[0,\xi_n]}(X_1)&= \int_0^{\xi_n}f(t) \dd t\le \| f_0-f\|_{1}+\xi_n f_0(0) \\&\le C\epsilon_n (\log n)^{1/\tau}+M\xi_n \le 2M\xi_n. \end{align*}
%where $\bar{M}= 2C$ and $\bar\tau = 1/\tau+1/3$.
 Bernstein's inequality gives
\[ I\le 2 e^{c_2 n\epsilon_n^2} \exp\left(-\frac14 \frac{v_n^2}{2M\xi_n + v_n/\sqrt{n}}\right). \]
If we take  $\eta_n = n^{-2/9}(\log n)^\beta$ , then
\[  \frac{v_n^2}{2M\xi_n + v_n/\sqrt{n}} \gtrsim n^{1/3} (\log n)^{2\beta}.
 \]

This tends to infinity faster than $n\eps_n^2 = n^{1/3}(\log n)^{2/3}$ whenever $2\beta>2/3$, i.e.\ when $\beta >1/3$.
\end{proof}

\begin{rem}
The derived rate is not the optimal but cannot be easily improved upon with the present type of proof. At first sight, one may wonder whether the tests $\Phi_n^+(0)$ can be improved upon by choosing different sequences $\{\tilde{c}_n\}$ and $M_n, \xi_n$. Unfortunately, the choice of $\xi_n$ and $M_n$ cannot be much improved upon. To see this, for bounding $I$ with Bernstein's inequality we need that $v_n$ in \eqref{eq:v_n} is positive. Assume $\xi_n=n^{-\beta_1}$ and  $\eta_n=n^{-\beta_2}$ (up to $\log n$ factors), we must have $\beta_1\ge \beta_2$.
Hence this restriction leads to $v_n\asymp -\sqrt{n}(\tilde c_n+\xi_n \eta_n)$.

Define $b_n=\max(\eps_n(\log n)^{1/\tau},\xi_n)$. Then $ \EE_f \ind_{[0,\xi_n]}(X_1)\lesssim b_n$, we can bound $I$ by
\[2\exp\left(c_2 n^{1/3}(\log n)^{2/3}-\frac14 \frac{v_n^2}{b_n + v_n/\sqrt{n}}\right).\]
We have two cases according to sequence $b_n$.
\begin{enumerate}
\item $b_n=\xi_n$, implies $\beta_1\le 1/3$. We have $\frac{v_n^2}{b_n + v_n/\sqrt{n}}\asymp n\xi_n\eta_n^2=n^{1-\beta_1-2\beta_2}$
should tend to infinity faster than $n^{1/3}$, hence $\beta_1+2\beta_2\le 2/3$. By combine all restrictions, we derive that $\beta_2$ necessarily has to satisfy $1/6\le \beta_2\le 2/9$.
\item $b_n=\eps_n(\log n)^{1/\tau}$, implies $\beta_1> 1/3$. Then $\frac{v_n^2}{b_n + v_n/\sqrt{n}}\asymp n^{4/3}(\xi_n\eta_n)^2=n^{4/3-2\beta_1-2\beta_2}\ge n^{1/3}$ gives $\beta_1+\beta_2\le 1/2$. Hence $\beta_2<1/6$.
\end{enumerate}
Therefore,  $\eta_n$ can not go to zero faster than $n^{-2/9}(\log n)^\beta$. 

\end{rem}

\begin{rem}\label{rem:median_mean}
As pointwise consistency is proved in Theorem \ref{th:centralth}, Theorem 4 in \cite{Salomond} implies that  the posterior median is a consistent estimator at any fixed point. Moreover, the posterior median has the same converge rate $n^{2/9}(\log n)^\beta$. The consistency of the posterior mean is not clear now. However, the posterior mean of $f$ is a decreasing density function, which provides a convenient way for estimation. We use either mean or median estimator according to different purpose in the simulation study.
\end{rem}

%%%-----------------

\subsection{A difficulty in the proof of theorem 4 in \cite{Salomond}}
The construction of the tests $\{\Phi_n^+(0)\}$ in the proof of theorem \ref{th:centralth} is new. In \cite{Salomond} a different argument is used, which we now shortly review (it is given in section 3.3 of that paper).
First we give a lemma for the following discussion.
\begin{lem}\label{lem:bound-f_G}
 Let  $\Pi$ be the prior distribution on $\scr{F}$ that is obtained via (\ref{eq:scalemix}), where  $G \sim DP(G_0, \alpha)$ and $G_0$ satisfies
 there exists positive numbers  $\th_0, \overline{a}, \overline{k}$ such that $$g_0(\theta) \le \overline{k} e^{-\overline{a}/\theta} \quad \mbox{for all} \quad \th \in (0,\theta_0).$$
 Then for any $x$ (possibly sequence) in $(0,\th_0)$,
\[	\Pi\left(\{f:f(0)-f(x)\ge A\}\right) \le \frac{\overline{k}}{\overline{a} A} x e^{-\overline{a}/x}\quad\text{ for every }\,\, A>0.\]
\end{lem}
\begin{proof}
	By the mixture representation of decreasing function $f$, (\ref{eq:scalemix}), and Markov's inequality we have
\[ \Pi\left(\{f:f(0)-f(x)\ge A\}\right) = \Pi\left(\int_0^x\th^{-1}dG(\th)\ge A\right)\le
A^{-1}  \int_0^{x}\th^{-1}g_0(\th)d\th. \]
By assumption \ref{ass:base} this is bounded by
\begin{align*} \overline{k}A^{-1}\int_0^{x}\th^{-1}e^{-\overline{a}/\th}d\th&= \overline{k} A^{-1}\int_{1/x}^\infty u^{-1}e^{-\overline{a}u}du\\ &
\le\overline{k}A^{-1}x\int_{1/x}^\infty e^{-\overline{a}u}du
=\overline{k}(\overline{a} A)^{-1} x e^{-\overline{a}/x}.
\end{align*}
\end{proof}

Let $\{h_n\}$ be a sequence of positive numbers. Trivially, we have
\[ f(0)-f_0(0) = f(0)- f(h_n) + f(h_n)- f_0(0). \]
Since both $f$ and $f_0$ are nonincreasing, $f(h_n)\le f(x)$ and $f_0(0) \ge f_0(x)$, for all $x\in [0,h_n]$. Hence,
\[ f(0)-f_0(0) \le  f(0)- f(h_n) + f(x)- f_0(x),\qquad \text{for all}\quad x\in [0,h_n]. \]	
This implies
	\[ f(0)-f_0(0) \le  f(0)- f(h_n) + h^{-1}_n L_1(f,f_0). \]
Using this bound and define a new sequence $\tilde{\eta}_n$, we get
\begin{equation}\label{eq:arg-zero}
\begin{split} \EE_0 \Pi\left( f(0)-f_0(0) > C \tilde\eta_n \mid X^n\right) &\le  \EE_0 \Pi\left( f(0)-f(h_n) > C \tilde\eta_n/2 \mid X^n\right)\\& + \EE_0 \Pi\left( L_1(f,f_0) > C \tilde\eta_n h_n/2 \mid X^n\right). \end{split}
\end{equation}
Choose $\tilde\eta_n$ and $h_n$ such that $\tilde\eta_n h_n = 2 \eps_n$. Theorem 1 in \cite{Salomond}  implies that the second term on the right-hand-side tends to zero. We aim to choose $\tilde\eta_n$ such that the first  term on the right-hand-side in \eqref{eq:arg-zero} also tends to zero. This term can be dealt with using lemma \ref{lem:KLn}:
\begin{align*}	\EE_0 \Pi\left( f(0)-f(h_n) > C \tilde\eta_n/2 \mid X^n\right)& \le \PP_0(\scr{D}_n) + c_1^{-1}e^{c_2 n \eps_n^2}  \Pi\left( f(0)-f(h_n) > C \tilde\eta_n/2\right)\\ & = o(1) + c_1^{-1}e^{c_2 n \eps_n^2}  \Pi\left( f(0)-f(h_n) > C \tilde\eta_n/2\right).  \end{align*}
Using lemma \ref{lem:bound-f_G}, the second term on the right-hand-side can be bounded by
\[    \frac{2\overline{k}}{\overline{a}c_1 C} \frac{h_n}{\tilde\eta_n} e^{c_2 n \eps_n^2-\overline{a} h_n^{-1}} \asymp \frac{h_n^2}{\eps_n} e^{c_2 n \eps_n^2-\overline{a} h_n^{-1}}\]
Since $n\eps_n^2 = n^{1/3} (\log n)^{2/3}$, the right-hand-side in the preceding display tends to zero ($n\to \infty$) upon choosing $h_n^{-1} \asymp n^{1/3}  (\log n)^\beta$ and $\beta>2/3$. This yields
\[ \tilde\eta_n \asymp \eps_n h_n^{-1} \asymp (\log n)^{\beta +1/3}, \]
which unfortunately does not tend to zero. Hence, we do not see how the presented argument can yield pointwise consistency of the posterior at zero.

\subsection{Attempt to fix the proof by adjusting the condition on the base measure}\label{subsec:adj}
A natural attempt to fix the argument consists of changing the condition on the base measure.
If the assumption on $g_0$ would be replaced with
\begin{equation}\label{eq:ch-priormass} \underline{k} e^{-\underline{a}/\th^\ga} \le g_0(\theta) \le \overline{k} e^{-\overline{a}/\theta^\ga} \quad \mbox{for all} \quad \th \in (0,\theta_0),\end{equation} then lemma \ref{lem:bound-f_G} would give the bound
\[	\Pi\left(\{f:f(0)-f(x)\ge A\}\right) \le \frac{\overline{k}}{\overline{a} A} x e^{-\overline{a}/x^\ga}.\]
Now we can repeat the argument and check whether it is possible to choose $\ga$ and $\{h_n\}$ such that
both $\tilde\eta_n \to 0$ and
\begin{equation}\label{eq:ch-otherprior} \frac{h_n^2}{\eps_n} e^{c n \eps_n^2-\overline{a} h_n^{-\ga}} = o(1) \end{equation}
hold true simultaneously. The requirement $\tilde\eta_n \to 0$ leads to taking $h_n = n^{-1/3} (\log n)^{\tilde\beta}$, with $\tilde\beta >1/3$. With this choice for $h_n$, equation \eqref{eq:ch-otherprior} can only be satisfied if $\ga >1$. Now if we assume \eqref{eq:ch-priormass} with $\ga>1$, then we need to check whether lemma \ref{lem:KLn} is still valid.
This is a delicate  issue  as we need to trace back in which steps of its proof the assumption on the base measure is used. In appendix B of \cite{Salomond} it is shown that the result in lemma \ref{lem:KLn} follows upon proving that
\begin{equation}\label{eq:kln} \Pi(\scr{S}_n) \ge \exp\left(-c_1 n \eps_n^2\right), \end{equation}
with $\eps_n=(\log n/n)^{1/3}$ (as in the statement of the lemma). Here, the set $\scr{S}_n$ is defined as
\[   \scr{S}_n=\left\{f:KL(f_{0,n},f_n)\le\epsilon_n^2,\int f_{0,n}(x)\left(\log\frac{f(x)}{f_0(x)}\right)^2dx\le\epsilon_n^2,\int_0^{\th_n} f(x)dx\ge 1-\epsilon_n^2\right\},\]
where
$$  \th_n=F_0^{-1}(1-\epsilon_n/(2n)), \qquad f_n(\cdot)=\frac{f(\cdot)I_{[0,\theta_n]}(\cdot)}{F(\th_n)}, \qquad  f_{0,n}(\cdot)=\frac{f_0(\cdot)I_{[0,\theta_n]}(\cdot)}{F_0(\th_n)}.$$
In lemma 8 of \cite{Salomond} it is proved that $\Pi(\scr{S}_n) \gtrsim \exp\left( -C_1 \eps_n^{-1} \log \eps_n\right)$ for some constant $C_1>0$, which implies the specific rate $\eps_n$. The proof of this lemma is rather complicated, the key being to establish the existence of a set  $\scr{N}_n \subset \scr{S}_n$ for which
 $\Pi(\scr{N}_n) \gtrsim \exp\left( -C_1 \eps_n^{-1} \log \eps_n\right)$.  Next, upon tracking down at which place the prior mass condition is used for that result (see appendix \ref{aped:kln}), we find  that it needs to be such that
 \begin{equation}\label{eq:required} \sum_{i=1}^{m_n} \log G_0(U_i) \gtrsim \eps_n^{-1} \log \eps_n \end{equation}
where $m_n \asymp \eps_n^{-1}$ and $U_i = (i\eps_n, (i+1) \eps_n]$ (see in particular inequality \eqref{eq:track-low-mass} in the appendix). Now assume \eqref{eq:ch-priormass}, then
\[ G_0(U_i) \ge \underline{k} \int_{U_i} e^{-\underline{a}/\th^\ga} \dd \th \ge \underline{k} \eps_n \exp\left(-\underline{a} (i\eps_n)^{-\ga}\right) \]
Hence
\begin{align*}  \sum_{i=1}^{m_n} \log G_0(U_i) &\gtrsim \log \underline{k}+ \eps_n^{-1}\log \eps_n -\underline{a} \sum_{i=1}^n (i\eps_n)^{-\ga} \\&\gtrsim \log \underline{k}+ \eps_n^{-1}\log \eps_n -\eps_n^{-\ga}, \end{align*}
if $\ga>1$ (which we need to assume for \eqref{eq:ch-otherprior} to hold). From this inequality we see that \eqref{eq:required} can only be satisfied if $\ga \in (0,1]$. We conclude that with the line of proof in \cite{Salomond} the outlined problem in the proof of consistency near zero cannot be fixed by adjusting the prior to \eqref{eq:ch-priormass}: one inequality requires $\ga>1$, while another inequality requires $\ga \in (0,1]$ and these inequalities need to hold true jointly.

%%%%%%%%

%-----------------------------

%%%%%%%%%%%%%%%%%%%%%%%%%%
\section{Gibbs Sampling in the DPM model}
\label{sec:alg}
Since a decreasing density can be represented as a scale mixture of uniform densities (see (\ref{eq:scalemix})) and the mixing measure is chosen according to a Dirichlet process, the model is a special instance of a so-called Dirichlet Process Mixture (DPM) Model. Algorithms for drawing from the posterior in such models have been studied in many papers over the past two decades, a key reference being \cit{Neal}. Here we shortly discuss the algorithm coined ``algorithm 2" in that paper. We  assume $G_0$ has a density $g_0$ with respect to Lebesgue measure.

Let $\#(x)$ denote the number of distinct values in the vector $x$ and let $x_{-i}$ denote the vector obtained by removing the $i$-th element of $x$. Denote by $\vee(x)$ and $\wedge(x)$ the  maximum and minimum of all elements in the vector $x$ respectively.

The starting point for the algorithm is a construction to sample from the DPM model:
\begin{equation}
\label{eq:mod2}
\begin{split}
 Z:=(Z_1,\ldots, Z_n) &\sim CRP(\alpha)\\
\Th_1,\ldots, \Th_{\#(Z)} &\simiid G_0 \\
X_1,\ldots, X_n \mid \Th_1,\ldots, \Th_{\#(Z)}, Z_1,\dots, Z_n& \simind Unif(0, \Th_{Z_i}).
\end{split}
\end{equation}
Here  CRP$(\alpha)$ denotes the ``Chinese Restaurant Process'' prior, which is a distribution on the set of partitions of the integers $\{1,2,\ldots,n\}$. This distribution is most easily described in a recursive way. Initialize by setting $Z_1=1$. Next, given $Z_1,\dots, Z_i$, let $L_i=\#(Z_1,\ldots, Z_{i})$ and  set
\[
Z_{i+1}=\begin{cases} L_i+1	 & \text{with probability}\: \alpha/(i+\alpha)\\
k 	 & \text{with probability}\:N_k/(i+\alpha). \end{cases}
\]
where $k$ varies over $\{1,\ldots, L_i\}$ and  $N_k = \sum_{j=1}^i \ind\{Z_j=k\}$ is the number of current $Z_j$'s equal to $k$. In principle this process can be continued indefinitely, but for our purposes it ends after $n$ steps.
One can interpret  the vector $Z$ as a partitioning of  the index set $\{1,\ldots,n\}$ (and hence the data $X=(X_1,\ldots, X_n)$) into $\#(Z)$ disjoint sets (sometimes called ``clusters'').
For ease of notation, write  $\Theta=(\Theta_1, \ldots, \Theta_{\#(Z)})$.

An algorithm for drawing from the posterior of $(Z, \Th)$   is obtained by successive substitution sampling (also known as  Gibbs sampling), where the following two steps are iterated:
\begin{enumerate}
  \item sample $\Th \mid (X,Z)$;
  \item sample $Z \mid (X,\Th)$.
\end{enumerate}
The first step
entails sampling from the posterior within each cluster. For the $k-$th component of $\Theta$, $\Theta_k$, this means sampling from
\begin{equation}\label{eq:fthk}
f_{\Th_k\mid X,Z}(\theta_k\mid x,z) \propto f_{\Th_k}(\theta_k)\prod_{j:z_j=k}f_{X_j\mid \Th_k}(x_j\mid\theta_k)=
 g_0(\theta_k)\prod_{j:z_j=k}\psi(x_j\mid\theta_k).
\end{equation}

\medskip

Sampling $Z \mid (X,\Theta)$ is done by cycling over all $Z_i$ ($1\le i \le n$) iteratively.  For $i \in \{1,\ldots, n\}$ and $k \in \{1,\ldots, 1+\vee(Z)\}$ we have
\begin{align}
f_{Z_i\mid Z_{-i},X,\Theta}(k\mid z_{-i},x,\theta)
 &\propto f_{X_i\mid Z_i,Z_{-i},\Theta}(x_i\mid k,z_{-i},\theta)f_{Z_i\mid Z_{-i},\Theta}(k\mid z_{-i},\th) \nonumber\\&=f_{X_i \mid \Th_{Z_i}}(x_i \mid \th_{k}) f_{Z_i \mid Z_{-i}}(k\mid z_{-i})	\label{eq:fz}
\end{align}

The right-hand-side of this display equals
\begin{equation}\label{eq:updatestepsZi}
\begin{split}
	\frac{N_{k,-i}}{n-1+\alpha}\psi(x_i \mid \th_k) \qquad \qquad   & \text{if}\quad 1\le k \le \vee(Z),
	\\
	 \frac{\alpha}{n-1+\alpha} \int \psi(x_i \mid \th) d G_0(\th)  \qquad \qquad & \text{if}\quad k =1+ \vee(Z),
\end{split}
\end{equation}
where $N_{k,-i}=\sum_{j\in \{1,\ldots,n\}\setminus \{i\}} \ind\{Z_j =k\}$.  The expression for $k=1+\vee(Z)$ follows since in that case sampling from $X_i \mid \Theta_k$ boils down to sampling from the marginal distribution of $X_i$. Summarising, we have the following algorithm:
\begin{algorithm}
	\begin{algorithmic}
		\caption{Gibbs Sampling in DPM model}
		\State Initialise $Z,\Th$.
		\For {each iteration} 
		\For {$i=1,2,\dots,n$} 
		\State Update $Z_i$ according to \eqref{eq:fz},
		\State That is, set $Z_i$ equal to $k$ with probabilities proportional to those given in \eqref{eq:updatestepsZi}. 
		\EndFor
		\For {$k=1,\dots,\#(Z)$} 
		\State Update $\Theta_k$ by sampling from the density in  \eqref{eq:fthk}.  
		\EndFor
		\EndFor
		
	\end{algorithmic}
\end{algorithm}

It may happen that over subsequent iterations of the Gibbs sampler certain clusters disappear. Then $\#(Z)$ and $\vee(Z)$ will  not be the same. If this happens, the $\Th_j$ corresponding to the disappearing cluster is understood to be removed from the vector $\Th$ (because the cluster becomes ``empty'', the prior and posterior distribution of such a $\Th_j$ are equal). The precise labels do not have a specific meaning and are only used to specify the partitioning into clusters.

In this step we need to evaluate $\int \psi(x_i \mid \th) d G_0(\th)$. One option is to numerically evaluate this quantity for $i=1,\ldots, n$ (it only needs to be evaluated once). Alternatively, the ``no-gaps'' algorithm of \cit{mac-muller} or ``algorithm 8'' of \cit{Neal} can be used and refer for further details to these papers.

%%%%%%%%%%%%%%%%%%%%%%%%%%
\section{Review of existing methods for estimating the decreasing density at zero}
\label{sec:methods}
In this section we review some  consistent estimators for a decreasing density $f_0$ at zero that have appeared in the literature. These will be compared with the Bayesian method of this paper  using a simulation study in section \ref{sec:simul}.

\subsection{Maximum penalised likelihood}\label{sec:WS}
In \cit{WoodroofeSun}, the maximum penalised likelihood estimator is defined as the maximiser of the following penalised log likelihood function:
$$
\ell_{\alpha}(f)=\sum_{i=1}^n\log f(X_i)-\alpha n f(0).
$$
Here $\alpha\ge0$ is a (small) penalty parameter. This estimator has the same form as the maximum likelihood estimator (MLE), being piecewise constant with at most $n$ discontinuities.
For fixed $\alpha\ge0$, for ease of notation here let $x_1<\cdots<x_n<\infty$ denote the ordered observed values and
$$w_0=0 \quad \text{and}\quad w_k=\alpha+\gamma x_k, \quad k=1,\dots,n$$
where $\gamma$ is the unique solution of the equation
\[\gamma=\min_{1\le s\le n}\left\{1-\frac{\alpha s/n}{\alpha+\gamma x_s}\right\}.\]

Denote by $f^P(\alpha,\cdot)$ the penalized estimator with penalty parameter $\alpha$. Taking $\alpha<x_n$,  $f^{P}(\alpha,\cdot)$ is a step function with
$$f^{P}(\alpha,x)=f^{P}(\alpha,x_k),\quad \forall x_{k-1}<x\le x_k, \quad\forall k=1,\dots,n.$$
At zero it is defined by right continuity and for $x\not\in[0,x_n]$ as $f^{P}(\alpha,x)=0$. Here
$$f^{P}(\alpha,x_k)=\min_{0\le i<k}\max_{k\le j\le n}\frac{(j-i)/n}{w_j-w_i}.$$
Geometrically, for $k=1,2,\ldots,n$, $f^{P}(\alpha,x_k)$ is the left derivative of the least concave majorant of the empirical distribution function of the transformed data $w_i, i=1,\dots,n$ evaluated at $w_k$. Note that an alternative expression for $f^{P}(\alpha,0)$ is $(1-\gamma)/\alpha$ which can be easily calculated.

Theorem {\bf 4} in \cit{WoodroofeSun} states that
\[n^{1/3}\{f^P(\alpha_n,0)-f_0(0)\}\Rightarrow^d \sup_{t>0}\frac{W(t)-(c+\beta t^2)}{t}\]
where $\alpha_n=cn^{-2/3}$, $\beta=-f_0(0)f'_0(0)/2$ and $W(t)$ denotes the standard Brownian motion. In \cit{WoodroofeSun}, the theoretically optimal constant $c$ is determined by minimizing the expected absolute value of the limiting distribution $f^P$, resulting in $c=0.649\cdot\beta^{-1/3}$.

\subsection{Simple and `adaptive' estimators}
\label{sec:KulLop}
In \cit{Kulikov}, $f_0(0)$ is estimated by the maximum likelihood estimator $\hat{f}_n$ evaluated at a small positive (but vanishing) number: $\hat{f}_n(cn^{-1/3})$ for some $c>0$.  Of course, the estimator depends on the choice of the parameter $c$.

In \cit{Kulikov}, Theorem 3.1, it is shown that
\[A_{21}\left\{n^{1/3}(\hat{f}_n(cB_{21}n^{-1/3})-f_0(cB_{21}n^{-1/3}))+cB_{21}f'_0(0)\right\}\]
converges in distribution to $D_{R}[W(t)-t^2](c)$  when $n\to\infty$. Here $D_{R}[Z(t)](c)$ is the right derivative of the least concave majorant on $[0,\infty)$ of the process $Z(t)$, evaluated at $c$. Furthermore, $B_{21}=4^{1/3}f_0(0)^{1/3}|f'_0(0)|^{-2/3}$ and $A_{21}=\sqrt{B_{21}/f_0(0)}$.

Based on this asymptotic result, two estimators are proposed, denoted as $f^S$ and $f^A$ (`S' for simple, `A' for adaptive). The first is a simple one with $cB_{21}=1$, then $f^S(0)=\hat{f}_n(n^{-1/3})$. The second is $f^A(0)=\hat{f}_n(c^*B_{21}n^{-1/3})$, where  $c^*\approx 0.345$ is taken such that the the second moment of the limiting distribution is minimized. Of course, to really turn this into an estimator, $B_{21}$ has to be estimated. Details on this are presented in section \ref{sec:simulstud}.

\subsection{Histogram estimator}\label{sec:hist}
In chapter 2 of \cit{GroJoCUP} a natural and simple histogram-type estimator for $f_0(0)$ is proposed. Let $\{b_n\}$ be a vanishing sequence of positive numbers and consider the estimator $f^H(0)=b_n^{-1}\mathbb{F}_n(b_n)$, where $\mathbb{F}_n$ is the empirical distribution of $X_1,\dots,X_n$. It can be shown that $\mbox{E} f^{H}(0)-f_0(0)$ behaves like $b_n f'_0(0)/2$ and the variance of $f^{H}(0)$ behaves like $f_0(0)/(nb_n)$ as $n\to\infty$. Then the asymptotic mean square error (aMSE) optimal choice for $b_n$ is $(2f_0(0)/f'_0(0)^2)^{1/3}n^{-1/3}=2^{-1/3}B_{21}n^{-1/3}$, where $B_{21}$ is as defined in the Section \ref{sec:KulLop}.

%%%%%%%%%%%%%%%%%%%
\section{Numerical illlustrations}
\label{sec:simul}

In this section we use the algorithm described in Section \ref{sec:alg} to sample from  the posterior distribution. We consider two data generating settings for the true density function: the standard Exponential distribution and the half-Normal distribution. Both densities are bounded, decreasing and satisfy assumption \ref{ass:f_0}. Suppose in the $j$-th iteration of the Gibbs sampler (possibly after discarding ``burn in'' samples) we have obtained $\left(\Th^{(j)}_{Z_1}, \ldots, \Th^{(j)}_{Z_n}\right).$
At iteration $j$, if the stationary region of the mcmc sampler has been reached, a sample from the posterior distribution is given by
\begin{equation}\label{eq:sampler}\hat{f}^{(j)}(x):= \frac1{n} \sum_{i=1}^n \psi_x\big(\Th^{(j)}_{Z_i}\big).\end{equation}
Two natural derived Bayesian point estimators are the posterior mean and the median.
Assuming $J$ iterations, a Rao-Blackwellized estimator for the posterior mean is obtained by computing $J^{-1}\sum_{j=1}^J \hat{f}^{(j)}(x)$ and an estimator for the posterior median at $x$ is the median value in $\{\hat{f}^{(j)}(x),\, j=1,\dots,J\}$. %As the posterior mean is a density function,
We implemented our procedures in {\bf Julia}, see \cit{bezanson17}. The computer code and datasets for replication of our examples forms part of the {\sf BayesianDecreasingDensity} repository (\url{https://github.com/fmeulen/BayesianDecreasingDensity}). For plotting we used functionalities of the {\bf{ggplot2}} package (see \cit{wickham09}) in {\bf R}. The computations were performed on a MacBook Pro, with a 2.7GHz Intel Core i5 with 8 GB RAM.

\subsection{Base measures}\label{subsec:basemeasures} 
To assess the influence of the base measure in the Dirichlet-process prior, we consider the following choices for the base measure:
\begin{enumerate}
 \item[(A)]
  The density  of the base measure vanishes exponentially fast near zero, as the  lower bound of Assumption \ref{ass:base} requires:
  \begin{equation}\label{eq:expont}g_0(\th)\propto e^{-\th-\th^{-1}}\ind_{[0,\infty)}(\th).\end{equation}

 \item[(B)]  The density of the $\mbox{Gamma}(2,1)$ distribution \[  g_0(\th)=\th e^{-\th}\ind_{[0,\infty)}(\th). \] 
 \item[(C)] The density of the $\mbox{Pareto}(\bar\alpha,\tau)$ distribution. That is
 \[ g_0(\th) = \bar\alpha \tau^{\bar\alpha} \th^{-\bar\alpha-1} \ind_{[\tau,\infty)}(\theta). \]
 Here, we consider various choices for the threshold parameter $\tau$.
 \item[(D)] The density is obtained as a mixture of the $\mbox{Pareto}(\bar\alpha,\tau)$  density, where the mixing measure on $\tau$ has the $\mbox{Gamma}(\lambda,\beta)$ distribution. This implies that $g_0(\th) \asymp \th^{\lambda-1}$ for $\th \downarrow 0$.  The parameter $\bar\alpha$ is fixed here, but could be equipped with with a ``hyper'' prior without adding much additional computational complexity.
\end{enumerate}
Note that cases (A), (B), (D)(when $\lambda>1$) satisfy Assumption \ref{ass:base} and case (C) does not.
In cases (A) and (B) the update on the  ``cluster centra'' $\th$
does not boil down to sampling from a ``standard'' distribution. In this case either rejection sampling or a Metropolis-Hastings step can be used, the details of which are  given in section \ref{sec:updating-theta} in the appendix.
In case (C) we have partial conjugacy, which in this case means that the $\th$'s can be sampled from a Pareto distribution.
Finally, case (D) can be dealt with by Gibbs sampling. More precisely, conditional on the current value of $\tau$, the $\th$'s can be sampled from the Pareto distribution just as in  case (C). Next, $\tau$ is sampled conditional on $(\th_1, \th_{\#z})$ from the density
\[ p(\tau \mid \th_1, \th_{\#z}) \propto p(\th_1, \th_{\#z}\mid \tau) p(\tau) \propto \tau^{\lambda + (\#z) \bar\alpha-1}e^{-\beta \tau} \ind\{\tau \le \min(\th_1,\ldots, \th_{\#z})\} \]
(where we use ``Bayesian notation'', to simplify the expressions). Hence, this boils down to sampling from a truncated Gamma distribution.

\subsection{Estimates of the density for two simulated datasets}
\label{sec:prior-influence}

\begin{figure}[!htp]
\begin{center}
	\includegraphics[scale=0.75]{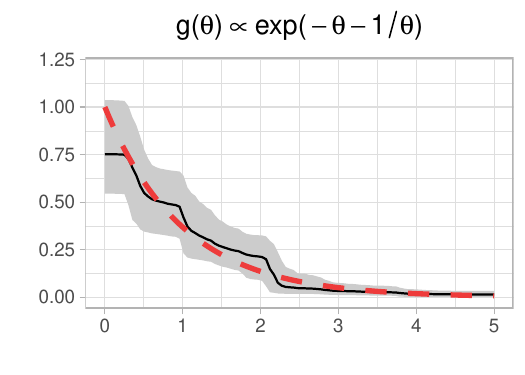}
	\includegraphics[scale=0.75]{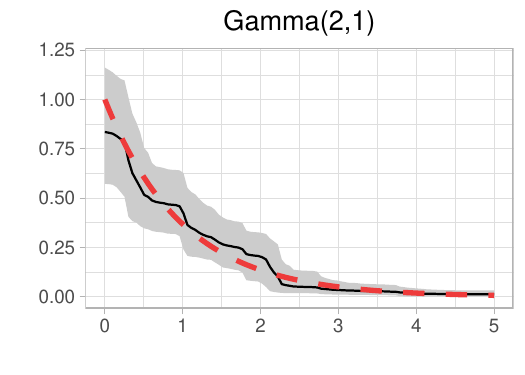}
	\includegraphics[scale=0.75]{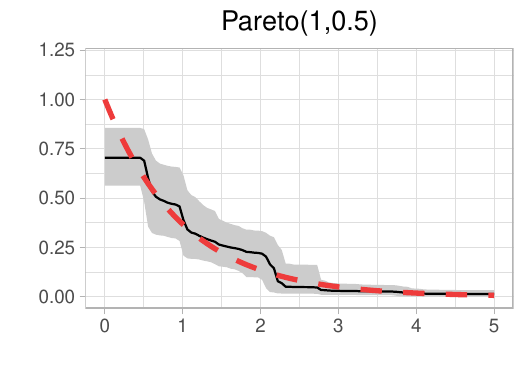}
	\includegraphics[scale=0.75]{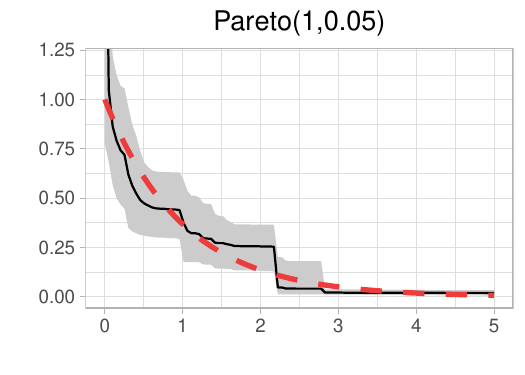}
	\includegraphics[scale=0.75]{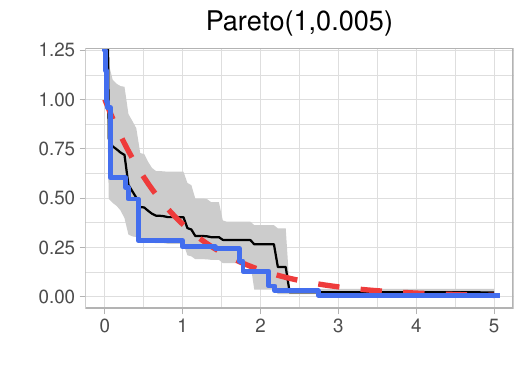}
	\includegraphics[scale=0.75]{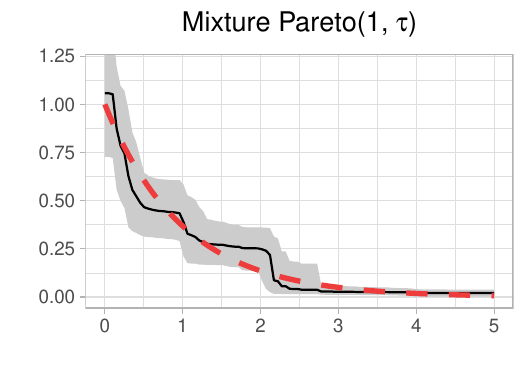}
\end{center}
\caption{In each panel the same dataset was used, which is a sample of size $100$ from the standard Exponential distribution. The black curve is the posterior mean and the shaded grey area depicts pointwise $95\%$	credible intervals. The dashed red curve is the true density.   The title in each of the figures refers to the base measure. In the mixture Pareto case, the mixing measure on $\tau$ was taken to be the $\mbox{Gamma}(2,1)$ distribution. In the lower left figure, the solid blue step-function is the maximum likelihood estimate. The inconsistency of this estimator at zero is clearly visible. Moreover, the figure suggests also inconsistency of the posterior mean when the base measure is taken to be the $\mbox{Pareto}(1,0.005)$ distribution.\label{fig:exp100}}
\end{figure}

\begin{figure}[!htp]
\begin{center}
	\includegraphics[scale=0.75]{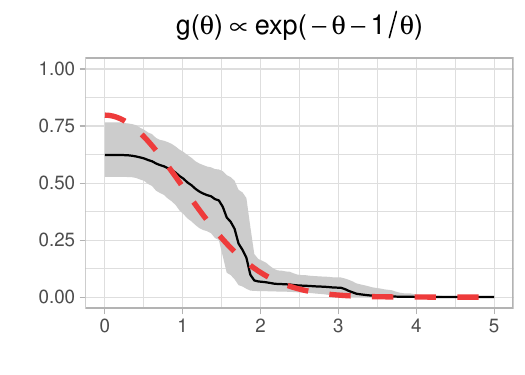}
	\includegraphics[scale=0.75]{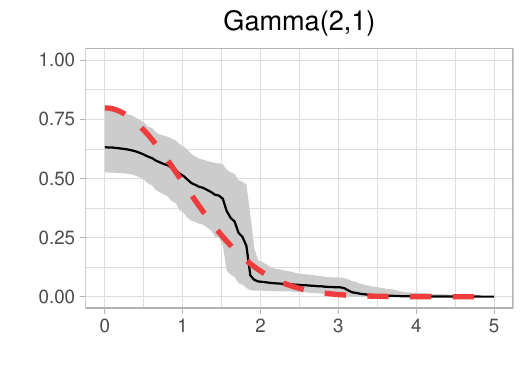}
	\includegraphics[scale=0.75]{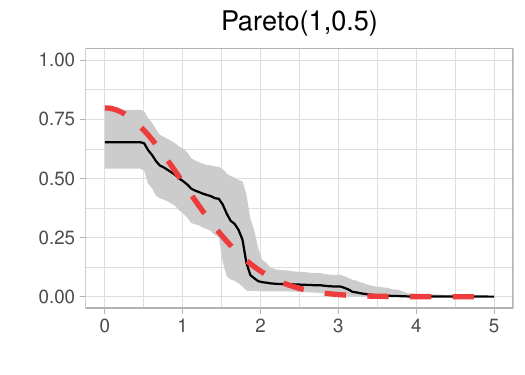}
	\includegraphics[scale=0.75]{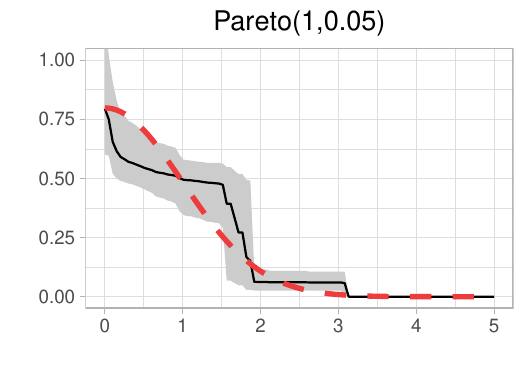}
	\includegraphics[scale=0.75]{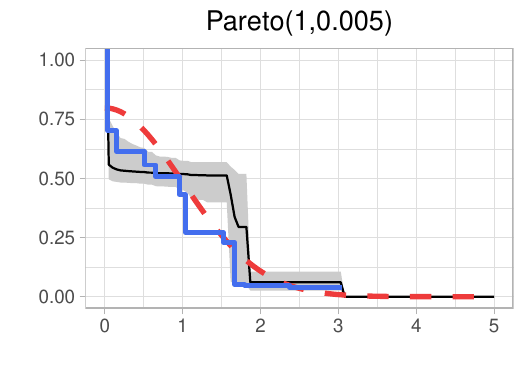}
	\includegraphics[scale=0.75]{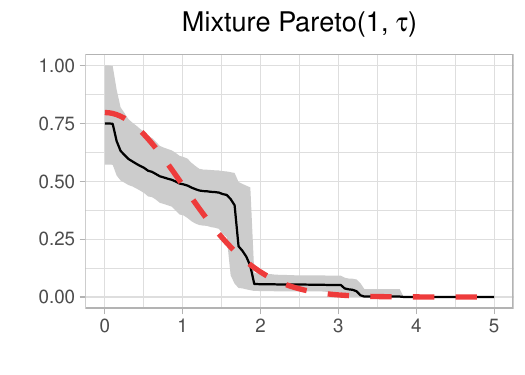}
\end{center}
\caption{Same experiment as in Figure \ref{fig:exp100}, this time with a sample of size $100$ from the halfNormal distribution. \label{fig:halfnormal100}}
\end{figure}

We obtained  datasets of size $100$ by sampling independently from both the standard Exponential distribution and the halfNormal distribution. In the prior specification, the concentration parameter $\alpha$ was fixed to $1$ in all simulations, while the base measure was varied over cases (A), (B), (C) with $\bar\alpha=1$, $\tau \in \{0.005, 0.05, 0.5\}$ and (D) with $\bar\alpha=1$, $\lambda=2$ and $\beta=1$. The algorithm was run  for $50.000$ iterations and the first half of the iterates were discarded as burn in. The computing time was approximately 2 minutes. In case Metropolis-Hastings steps were used for updating $\th$'s, the acceptance rates of the random-walk updates was approximately $0.35$, both in case (A) and (B). The results are displayed in figures \ref{fig:exp100} and \ref{fig:halfnormal100}. From the top figures we see that the posterior mean and pointwise credible bands visually look similar for the choices of base-measure under (A) and (B). If the base measure is chosen according to (C), the middle and bottom-left  figures show the effect of the parameter $\tau$. Choosing $\tau$ too small (here: $0.005$) the posterior mean appears inconsistent at zero, similar as the Grenander estimator which is added to the figure for comparison. For somewhat larger values of $\tau$ (middle-left figure), the estimate near zero is like a histogram estimator. Finally, the bottom-right figure shows the posterior mean under the base measure specification (D). Here, the posterior mean looks comparable as obtained under  (A) and (B), suggesting that we are able to learn the parameter $\tau$ from the data. In fact, whereas the prior mean of $\tau$ equals $2$, the average of the non burn in samples of $\tau$ equals $0.66$.  We have repeated the whole experiment with sample size $1000$. The results are in Appendix \ref{sec:n=1000results_exp1}. 

\subsection{Distribution of the posterior mean for $f(0)$ under various bases measures.}

In this section we compare base measures (A), (B) and (D) for estimating $f$ at zero.  In the experiment, we considered samples of sizes either $50$ or $250$. We computed the posterior mean for $f(0)$ for each sample based on $10,000$ MCMC-iterations, discarding the first half as burnin.  The Monte-Carlo sample size was taken equal to $500$. Figure \ref{fig:simf0} summarises the results. While the density for base measure (D) is slightly more spread, contrary to base measures (A) and (B), it concentrates on  correct values for both the Exponential and HalfNormal distribution.

\begin{figure}[!htp]
\centering
\includegraphics[scale=0.7]{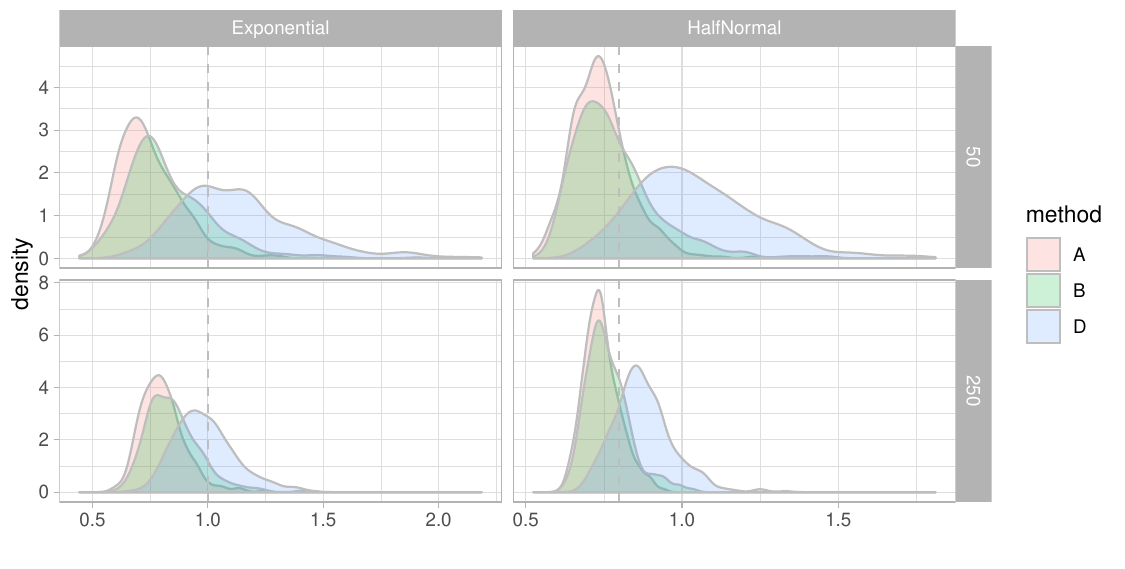}
\caption{Posterior mean estimator for $f(0)$ for sample sizes $50, 250$ in case the true data-generating distribution is either standard Exponential or halfNormal. The posterior mean is computed by taking $10,000$ MCMC-samples and discarding the first $5,000$ as burnin samples. The Monte-Carlo sample size was taken equal to $500$. For the considered sample sizes, only method (D) concentrates around the correct values.  \label{fig:simf0} }
\end{figure}

\subsection{Empirical assessment of the rate of contraction}
We also performed a large scale experiment to empirically assess the rate of contraction of the posterior median at zero, under either choices (A), (B) or (D) for the base measure. Our proof for deriving the contraction rate really requires a base-measure as under (A) and now the underlying idea is to see in a simulation study whether $g_0(\th) \sim \th$ for $\th$ near $0$ is suitable or not.  In the experiment, we first fixed a sample size $n$ and generated $n$ independent realisations from the standard Exponential distribution. We then ran the MCMC sampler for $20.000$ iterations, and kept the final iterate for initialisation of all chains ran for that particular sample size. Next, we repeated $50$ times
\begin{enumerate}
  \item  sample a dataset of size $n$ from the standard Exponential distribution;
  \item run the MCMC algorithm for $2500$ iterations;
  \item compute the median value at zero obtained in those samples.
\end{enumerate}
The Metropolis-Hastings proposals for updating the $\th$'s were tuned such that the  acceptance rate was about $20\%$ in all cases.
If the averages are denoted by $y_1,\ldots, y_{100}$, we finally computed the Root Mean Squared Error, defined by
$ \sqrt{0.02 \sum_{i=1}^{50}(y_i - 1)^2}$. By repeating this experiment for all three choices of base measure and various values of $n$, we obtained figure \ref{fig:rate-comparison_new}. The contraction rate is an asymptotic property, and hence there is definitely uncertainty on which values of $n$ correspond to that. The computed slopes do not give a conclusive answer to the actual rate of contraction. For the halfNormal distribution, it is conceivable that methods (A) and (B) yield rate $n^{-1/3}$, whereas method (D) gives a rate almost $n^{-1/2}$. The latter can intuitively be explained by the fact that the slope of the density of the halfNormal is zero at zero which coincides with realisations from the prior.  For the Exponential distribution, methods (A) and (B) support rate $n^{-2/9}$, whereas method (D) has worse rates. For completeness, we tabulated the computed slopes in Table \ref{table:ratecomparision0}. The difficulty with rate-assessment by finite samples for Dirichlet mixture priors has been noted recently in  \cit{wehrhahn19} as well.

\begin{figure}[!htp]
\centering
\includegraphics[scale=0.65]{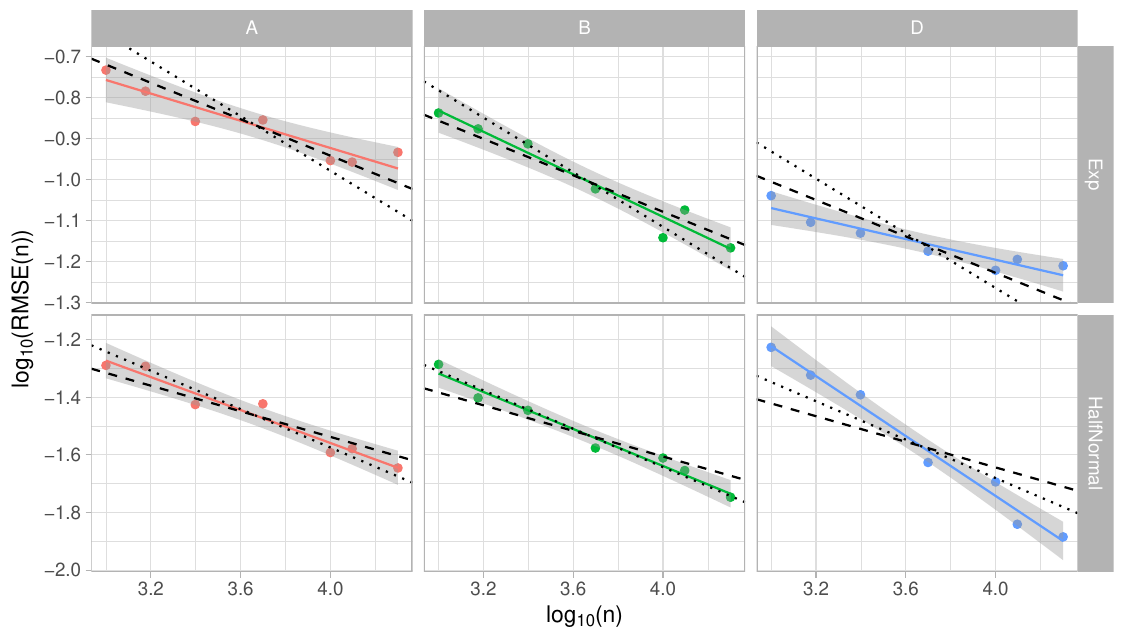}
\caption{ The base10-log of the RMSE versus the base10-log of the sample size under 3 different base measures (method A: $g(\th)\propto\exp(-\th-1/\th)$, method B: $g(\th)\propto\th\exp(-\th)$, method D: mixture of Pareto). Each dot corresponds to the average of the posterior means using Monte-Carlo size $50$. In each panel a least-squares fit is added along with a $95\%$-confidence interval. The dashed and dotted lines are best least squares fits with slopes $-2/9$ and $-1/3$ respectively.   }\label{fig:rate-comparison_new}
\end{figure}

\begin{table}[!htp]
\begin{tabular}{c |   | c | c | c }
           & \multicolumn{3}{c}{Method} \\ 
 & A & B & D   \\  \hline 
Exp &  $-0.166$   & $-0.260$   & $ -0.126$  \\
halfNormal &  $-0.286$ & $-0.321$  & $-0.520$  
\end{tabular}
\caption{Slopes of fitted lines in Figure \ref{fig:rate-comparison_new}. \label{table:ratecomparision0}}
\end{table}

\subsection{Comparing between Bayesian and  various frequentist methods for estimating $f_0$ at 0}
\label{sec:simulstud}
In this section we present a simulation study comparing our Bayesian estimator (posterior median) with various frequentist estimators available for $f_0(0)$ discussed in section \ref{sec:methods}.
We simulated 50 samples of sizes $n=50,200,10000$ from the standard exponential distribution and halfNormal distribution. For each sample, the following estimators are calculated: the posterior median estimator $f^B$, the penalized NPMLE $f^P$, the two estimators $f^S$ and $f^A$ and the histogram type estimator $f^H$. All these estimators require choosing some input parameters.
\begin{enumerate}
\item The posterior median estimator $f^B(0)$ is computed using the DPM prior with concentration parameter $\alpha=1$  and base measure in (\ref{eq:expont}). The total number of MCMC iterations was chosen to be $30000$, with $15000$ burn-in iterations. The posterior median was computed as median value of samples for $\hat{f}(0)$ in equation \eqref{eq:sampler}.
\item For the penalized estimator $f^P(\alpha_n,0)$ the parameter $\alpha_n=0.649\hat{\beta}_n^{-1/3}n^{-2/3}$ was taken with  \[\hat{\beta}_n=\max\left\{f^P(\alpha_0,0)\frac{f^P(\alpha_0,0)-f^P(\alpha_0,x_m)}{2x_m},n^{-1/3}\right\}.\]
   Here $x_m$ is the second point of jump of $f^P(\alpha_0,\cdot)$ and $\alpha_0=0.0516,0.0205$ for $n=50,200$ (listed in \cit{WoodroofeSun}).
\item For $f^S(0)=\hat{f}_n(n^{-1/3})$ no tuning is needed. For the other estimator we take $f^A(0)=\hat{f}_n(0.345\hat{B}_{21}n^{-1/3})$,  where
 \begin{equation}
 \label{eq:adaptB21}
 \hat{B}_{21}=4^{1/3}f^S(0)^{1/3}|\hat{f}'_n(0)|^{-2/3},\end{equation}
a consistent estimator of $B_{21}$ where \[\hat{f}'_n(0)=\min\{n^{1/6}(\hat{f}_n(n^{-1/6})-\hat{f}_n(n^{-1/3})),-n^{-1/3}\}.\]
\item For the histogram estimator  $f^H(0)=\mathbb{F}_n(\hat{b}_n)/\hat{b}_n$,  $\hat{b}_n=2^{-1/3}\hat{B}_{21}n^{-1/3}$ was chosen with $\hat{B}_{21}$ as in (\ref{eq:adaptB21}).\\
\end{enumerate}

Figure \ref{fig:boxexp} shows, for each combination of sample size and estimation method described, the boxplots of the 50 realized values based on samples from the standard exponential distribution.  Figure \ref{fig:boxhfm} shows these boxplots for the samples from the halfNormal distribution.

\begin{figure}[!htp]
\includegraphics[width=0.9\textwidth]{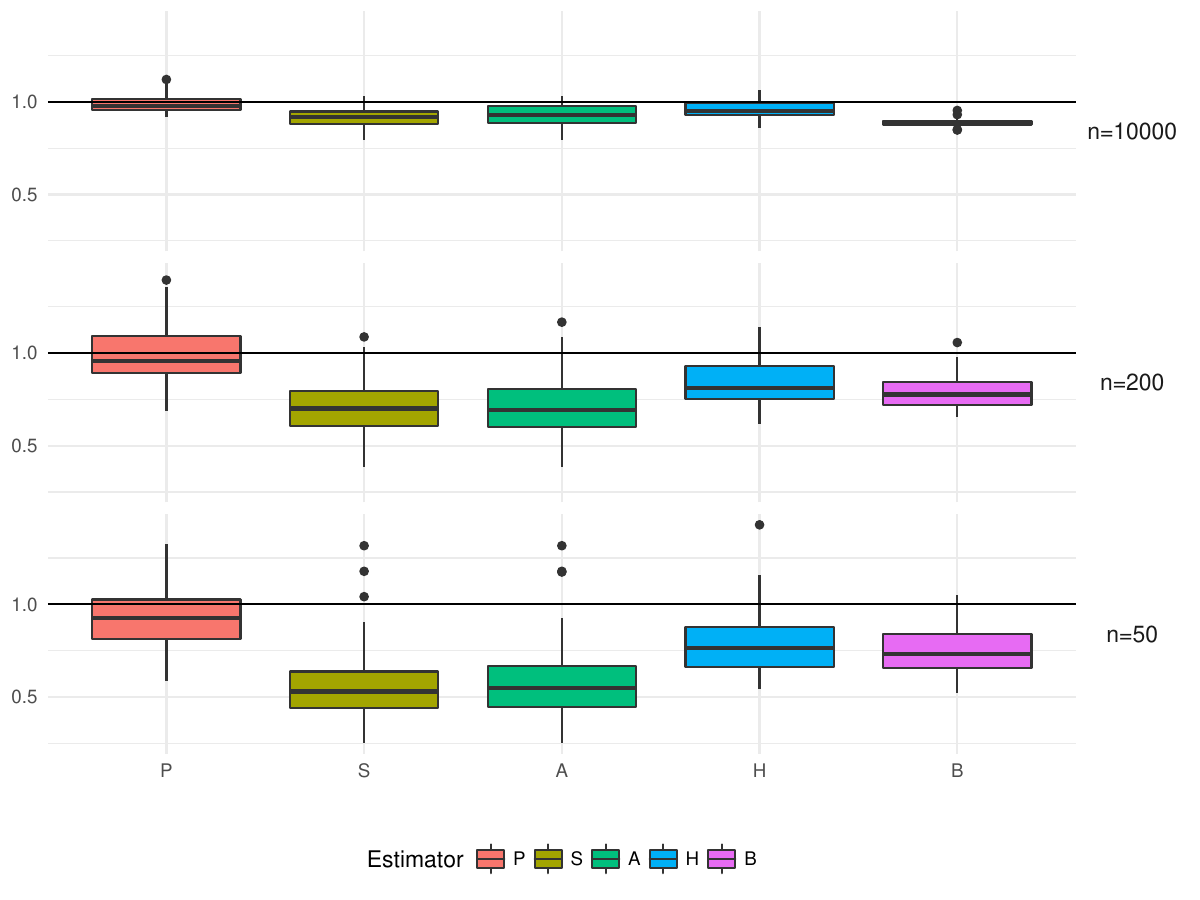}

\caption{Boxplots based on $50$ replications, where a sample of size $n$ is drawn from the standard exponential distribution.
Here $P$,$S$,$A$,$H$,$B$ correspond to the penalized maximum likelihood-, simple-, adaptive-, histogram- and posterior median- estimator respectively. The horizontal lines indicate the true value of $f_0(0)=1$.}
\label{fig:boxexp}
\end{figure}

\begin{figure}[!htp]

\includegraphics[width=0.9\textwidth]{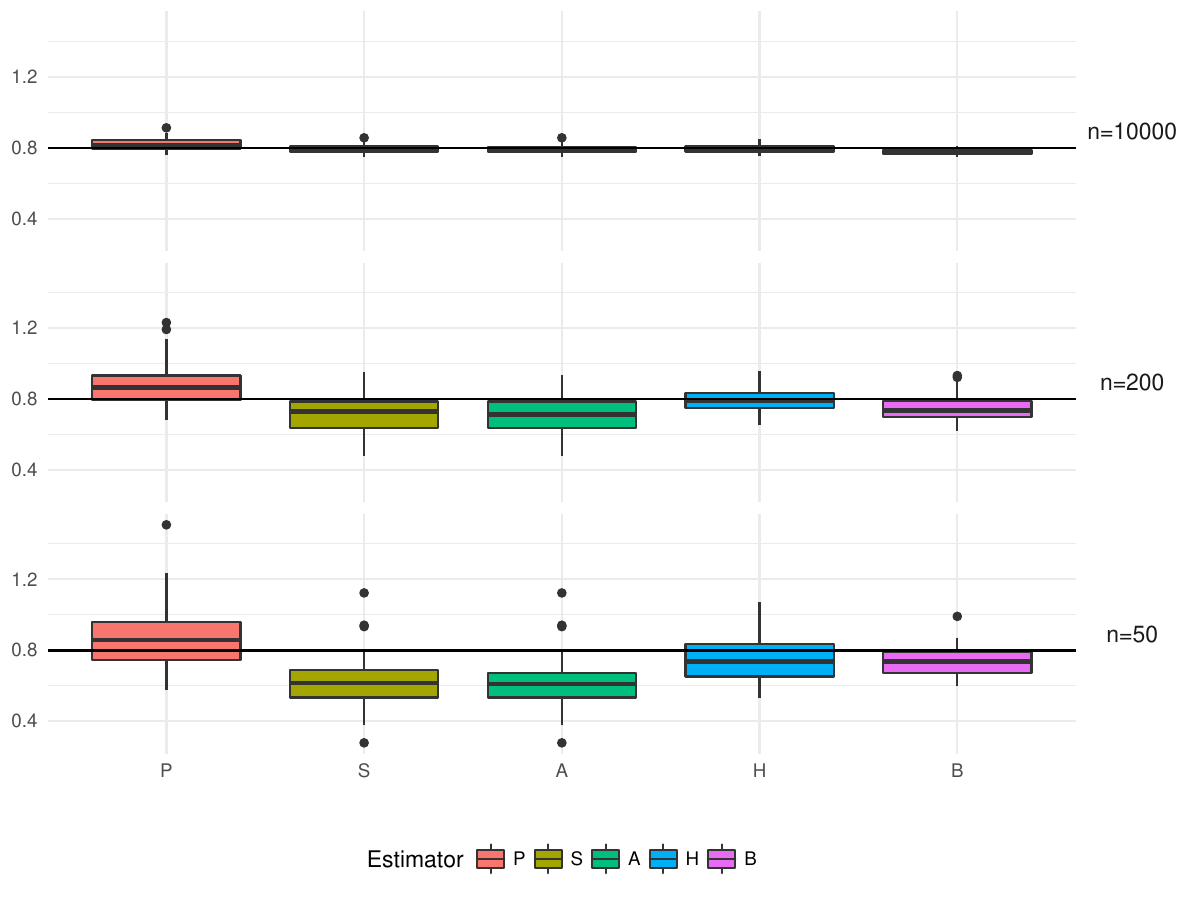}

\caption{Boxplots based on $50$ replications, where a sample of size $n$ is drawn from the halfNormal distribution. The rows correspond to the sample sizes $n=50,$ $200$ and $10000$.
Here $P$,$S$,$A$,$H$,$B$ correspond to the penalized maximum likelihood-, simple-, adaptive-, histogram- and posterior median- estimator respectively. The horizontal lines indicate the true value of $f_0(0)=\sqrt{2/\pi}$.}
\label{fig:boxhfm}
\end{figure}

\medskip
In table \ref{tab:exptab} we compare the bias, variance and mean squared error of these consistent estimators based on data from the standard exponential distribution. For the standard exponential data, the penalized estimator $f^{P}(0)$ performs best in the MSE sense. The Bayesian estimator $f^{B}$ has smallest variance, but big bias when the sample size is large ($n=10000$). This might be explained by the small contraction rate $n^{-1/6}$ at zero, but also by the fact that the Bayesian method is not specifically aimed at only estimating the density at zero, but instead the full density.
\begin{table}[!htp]
\centering
\begin{tabular}{ccccccc}
\hline
\hline
$n$ &  & $f^{P}$ & $f^{S}$ & $f^{A}$ & $f^{H}$ &$f^{B}$ \\[0.5ex]
\hline
\multirow{3}{*}{50} &Bias &-0.067& -0.423 &-0.402 &-0.214 &-0.266\\
&Var &0.033 &0.042 &0.049 &0.030 &0.013 \\
&MSE &0.037& 0.222 &0.210 &0.076& 0.084  \\
[0.8ex]
\multirow{3}{*}{200} &Bias &-0.001 &-0.286 & -0.271& -0.158 & -0.221\\
&Var &0.029 &0.020 &0.027 &0.015 &0.007\\
&MSE &0.029 &0.101 &0.100 &0.040 &0.056\\
[0.8ex]
\multirow{3}{*}{10000} &Bias &-0.011 & -0.084  &-0.072 & -0.041  &-0.112 \\
&Var &0.002 &0.002 &0.003 &0.002 &0.0004\\
&MSE &0.002 &0.010 &0.009 &0.004 &0.013\\
\hline
\end{tabular}
\caption{Simulated bias, variance and mean squared error for the five estimators from standard exponential distribution.}\label{tab:exptab}
\end{table}

 Table \ref{tab:halfnormtab} lists the bias, variance and MSE values of the estimators with observations sampled from the halfNormal distribution. For the halfNormal data, the histogram estimator $f^{H}$ behaves best in the bias and MSE sense. This can probably  be explained by the behaviour of $f_0$ near zero, note  that $f'_0(0)=0$ in the halfNormal case. The estimator for $f'_0(0)$, $\hat{f}'_n(0)$, probably quite unstable which leads to big value for $\hat{B}_{21}$ resulting in a big bandwidth $\hat{b}_n$.  As the behaviour of the underlying density is ``flat" near zero, the MSE-optimal choice of bandwidth is of the slower order $n^{-1/5}$. The posterior mean again has smallest variance.

\begin{table}[!htp]
\centering
\begin{tabular}{ccccccc}
\hline
\hline
$n$ &  & $f^{P}$ & $f^{S}$ & $f^{A}$ &$f^{H}$ & $f^{B}$\\[0.5ex]
\hline
\multirow{3}{*}{50}
&Bias & 0.063 & -0.182&  -0.185& -0.043&  -0.073 \\
&Var &0.029& 0.022& 0.022& 0.016& 0.007 \\
&MSE &0.033& 0.055& 0.056& 0.018& 0.012\\
[0.8ex]
\multirow{3}{*}{200}
&Bias &0.080& -0.086& -0.088& -0.011& -0.051 \\
&Var & 0.014& 0.012& 0.012& 0.004& 0.005\\
&MSE &0.020& 0.019& 0.020 & 0.004& 0.008\\
[0.8ex]
\multirow{3}{*}{10000} &Bias &0.0216&  -0.0022&  -0.0060& -0.0019&  -0.0239\\
&Var &0.0010& 0.0005& 0.0006& 0.0005& 0.0002\\
&MSE &0.0015& 0.0005& 0.0006& 0.0005& 0.0008\\
\hline
\end{tabular}
\caption{Simulated bias, variance and mean squared error for the five estimators based on samples from the standard halfNormal distribution.}\label{tab:halfnormtab}
\end{table}

\subsection{Application to fertility data}
\label{sec:fertility}
In \cit{Keiding} data concerning the fertility of a population are analysed. The aim is to estimate the distribution of the duration for women to become pregnant from when they start attempting, based on data from so-called current durations. These current durations can be modeled as described in the introduction. Indeed, the true durations are modeled as sample from an unknown distribution function $H_0$. According to length-biased sampling, individuals are selected and then the time since the start of attempting to become pregnant is administered. This is called the current duration, and can be seen as a uniform random fraction of the true duration of the selected individual. This current duration then has bounded decreasing probability density $f_0$ as given in (\ref{eq:expre}). The distribution function of the durations $H_0$, can be expressed in terms of $f_0$ as in Equation(\ref{eq:invrel}).  For more information on the design of this study we refer to \cit{Keiding}. For illustration purpose we only used the $n=618$ measured current durations that do not exceed $36$ months. Figure \ref{fig:hist_f} shows the histogram of 618 raw data, modeled as sample from the decreasing density $f_0$.
\begin{figure}
	\begin{center}
		\includegraphics[scale=0.7]{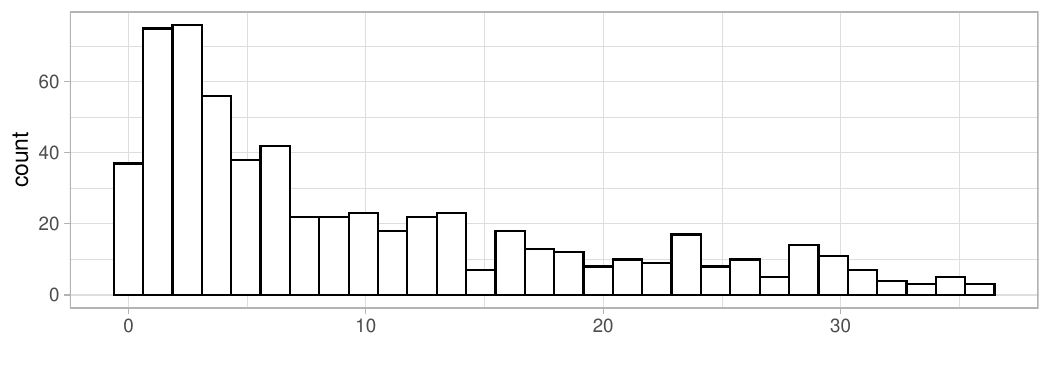}
		\caption{Histogram of the current durations fertility data that do not exceed $36$ months.  \label{fig:hist_f}}
	\end{center}
\end{figure}

In this section we estimate the density $f_0$ using base measure choice (A) which satisfies assumption \ref{ass:base}
 and (D) which does not satisfy assumption \ref{ass:base} with concentration parameter $\alpha=1$. Then each MCMC iterate of the posterior mean can be converted to an iterate for $H_0$ using the relation \eqref{eq:invrel}. In \cit{GroJo15} chapter 9, pointwise confidence bands for $f_0$ and $H_0$ are constructed based on the smoothed maximum likelihood estimator. Having derived the estimators, producing such confidence bands needs quite some fine tuning.  In this section, we construct the Bayesian counterpart of the confidence bands, credible regions for $H_0$. Contrary to the frequentist approach, having the machinery available for computing the posterior mean, the pointwise credible sets can be obtained directly from the MCMC output. The results for the fertility data are shown in Figures \ref{fig:fertdata} using base measures (A) and (D) respectively.

\begin{figure}
\begin{center}
\includegraphics[scale=0.75]{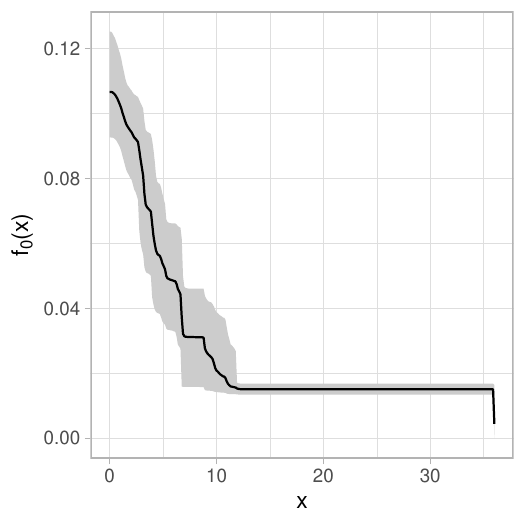}
\includegraphics[scale=0.75]{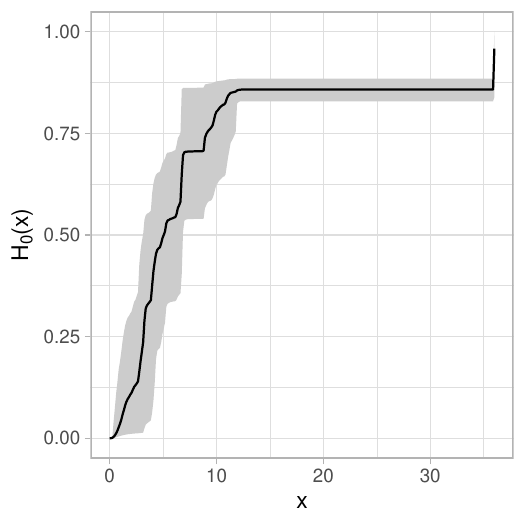}

\includegraphics[scale=0.75]{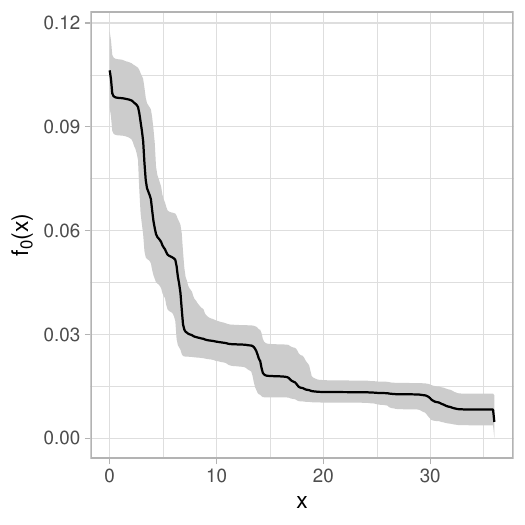}
\includegraphics[scale=0.75]{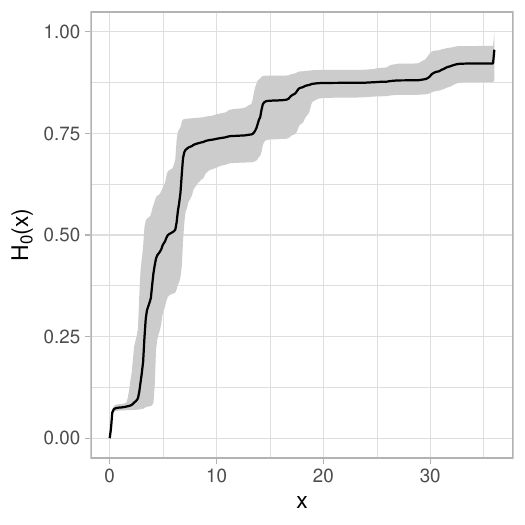}

\caption{Fertility data. Top: results for base measure (A) and $\alpha=1$. Bottom: results for base measure (D) and $\alpha=1$. Left: posterior mean and $95\%$ pointwise credible sets for probability density function $f_0$. Right: corresponding estimate and pointwise credible sets for the distribution function $H_0(x)=1-f_0(x)/f_0(0)$. \label{fig:fertdata}}
\end{center}
\end{figure}

\section{Discussion}
In this paper we have used Bayesian analysis to nonparametrically estimate a decreasing density based on a random sample. Particular emphasis is given to estimation of the density at zero and  sufficient criteria on the base measure of the prior are derived to obtain contraction rate $n^{-2/9}$. Besides a base measure attaining this rate, we have investigated the relative performance of   other base measures by means of a Monte Carlo study. This study was extended to compare multiple frequentist estimators for estimating the density at zero to a Bayesian derived point estimator.  

It remains an open question whether for a given density function $f$ there exists a base measure such that the contraction rate for estimation of $f(0)$ is $n^{-1/3}$. From the simulation study it appears that taking a mixture of Pareto densities as base measure empirically yields satisfactory performance and henceforth we  recommend taking base measure (D) from Section \ref{subsec:basemeasures}.

\appendix

\section{Review and supplementary proof of inequality \eqref{eq:kln}}\label{aped:kln}

In this section we point out a technical issue arising in the proof of  inequality \eqref{eq:kln}. As mentioned in  section \ref{subsec:adj}, it suffices to lower bound the prior mass of a certain subset $\scr{N}_n$ of $\scr{S}_n$, for which lower bounding $\Pi(\scr{N}_n)$ is tractable. To construct this set, we first need some approximation results.

\begin{lem}\label{lem:mix}
	For any $\th_0>0$ there exists a discrete measure $\tilde{P}=\sum_{i=1}^{\tilde{N}} \tilde{p}_i\delta_{y_i}$, with $y_i\in [\th_0,\infty)$, $p_i\in [0,1]$, $\tilde{N}\lesssim 1/\eps_n$  and  $\sum_{i=1}^{\tilde{N}} p_i=\int_{\th_0}^\infty f_0(x)dx$ such that
	\[ \int_{\th_0}^\infty \left(\sqrt{f_0(x)}-\sqrt{f_{\tilde{P}}(x)}\right)^2 \dd x \lesssim \eps_n^2. \]
	Moreover, the sequence $\{y_i\}$ can be taken such that $|y_i - y_j| \geq 2\eps_n^2$ for all $i, j \le \tilde{N}$. 	
\end{lem}
\begin{proof}
	Without the claimed separation property, existence of the discrete measure follows from lemma 11 in \cite{Salomond}. Denote this measure by $P=\sum_{i=1}^N p_i \delta_{z_i}$ and note that $N\lesssim 1/\eps_n$. The set $y_1,\ldots, y_{\tilde{N}}$ is obtained from $\{z_1,\ldots, z_N\}$  by removing points from the latter set which are not $2\eps_n^2$-separated. Clearly, $\tilde{N}\le N \lesssim 1/\eps_n$.  The mass $p_i$ of any removed point $z_i$ is subsequently  added to the point $y_j$ ($1\le j \le \tilde{N}$) that is closest to $z_i$. Denote the mass of $y_j$, obtained in this way, by $\tilde{p}_j$.  Hence, we can written $\tilde{P}=\sum_{j=1}^{\tilde{N}}\tilde{p}_j\delta_{y_j}=\sum_{i=1}^{N}p_i\delta_{y_{k(i)}}$, where $k(i)=j$ if $p_i$ assigned to $\tilde{p}_j$. Furthermore,
	\begin{align*}
	L_1\left( f_{P}, f_{\tilde{P}}\right)
	&=\int \Big| \sum_{i=1}^N p_i\psi_x(z_i)-\sum_{j=1}^{\tilde{N}}\tilde{p}_j\psi_x(y_j)\Big| \dd x\\
	&=\int \Big| \sum_{i=1}^N p_i(\psi_x(z_i)-\psi_x(y_{k(i)}))\Big| \dd x\\
	&=\int \Big| \sum_{i:z_i\neq y_{k(i)}} p_i(\psi_x(z_i)-\psi_x(y_{k(i)}))\Big| \dd x
	\end{align*}
	Since for any $\th_0<\th_1< \th_2$,
	\begin{align*}
	\int \mid\psi_x(\th_1)-\psi_x(\th_2)\mid dx&= \int_{x\le \th_1}+\int_{\th_1<x\le \th_2}+\int_{x>\th_2}\mid\psi_x(\th_1)-\psi_x(\th_2)\mid dx\\
	&=2(\th_2-\th_1)/\th_2\lesssim \th_2-\th_1.\end{align*}
	This implies that
	\begin{align*}L_1\left( f_{P}, f_{\tilde{P}}\right)&\le \sum_{i:z_i\neq y_{k(i)}}p_i \int \mid \psi_x(z_i)-\psi_x(y_{k(i)})\mid d x\\
	&\le \sum_{i:z_i\neq y_{k(i)}}p_i \eps_n^2
	\lesssim \eps_n^2
	\end{align*}
	The claimed result now follows from the triangle inequality and that the squared Hellinger distance is bounded by the $L_1$-distance.
\end{proof}

\begin{lem}\label{lem:mix2}
	Assume $f_0$ satisfies assumption \ref{ass:f_0}. There exists a discrete probability measure $\tilde{P}$, supported on $\{i\eps_n,\, 1\le i \le N'\} \cup \{y_j,\, 1\le j \le \tilde{N}\}$, with $N' = \lfloor x_0/\eps_n\rfloor $ such that
	\[  \int_0^\infty \left( \sqrt{f_0(x)} - \sqrt{f_{\tilde{P}}(x)} \right)^2 \dd x \lesssim \eps_n^2. \]
	
\end{lem}

\begin{proof}
	By lemma \ref{lem:mix} applied with $\th_0=x_0$ it suffices to prove $\int_0^{x_0} (\sqrt{f_0(x)}-\sqrt{f_{\tilde{P}(x)}})^2 \dd x\lesssim \eps_n^2$.
	Define the measure $\tilde{P}=\sum_{i=1}^{N'}p'_i\delta_{i\eps_n}+\sum_{j=1}^{\tilde{N}}\tilde{p}_{j}\delta_{y_j}$, where $\tilde{p}_{j}$ is as defined in lemma \ref{lem:mix} and
	$$p'_i=\begin{cases}(f_0((i-1)\eps_n)-f_0(i\eps_n))i\eps_n &\quad \text{if}\quad i< N'\\(f_0((N'-1)\eps_n)-a)N'\eps_n &\quad \text{if}\quad i=N'\end{cases}$$
	with $a=\sum_{j=1}^{\tilde{N}}\tilde{p}_{j}/y_j$. Then  for  $x\in ((i-1)\eps_n,i\eps_n]$,
	\begin{align*}
	f_{\tilde{P}}(x) & = \sum_{k=i}^{N'} p'_k \psi_x(k\eps_n) + \sum_{j=1}^{\tilde{N}} \tilde{p}_j \psi_x(y_j) = \sum_{k=i}^{N'} \frac{p'_k}{k\eps_n} + a  \\ & = \sum_{k=i}^{N'-1} k\eps_n \frac{f_0((k-1)\eps_n)-f_0(k\eps_n)}{k\eps_n} + \frac{f_0((N'-1)\eps_n)-a}{N'\eps_n} N'\eps_n + a \\&=f_0((i-1)\eps_n)
	\end{align*}
	By the mean value theorem, it follows that
	\begin{align*}
	\int_0^{x_0} \left(\sqrt{f_0(x)}-\sqrt{f_{\tilde{P}}(x)}\right)^2 \dd x  & = \sum_{i=1}^{N'} \int_{(i-1)\eps_n}^{i\eps_n}  \left(\sqrt{f_0(x)}-\sqrt{f_0((i-1)\eps_n)}\right)^2 \dd x \\
	& \le \sum_{i=1}^{N'} \int_{(i-1)\eps_n}^{i\eps_n}  \left(\frac{f_0'(\zeta_i)}{2\sqrt{f_0(\zeta_i)}}(x-(i-1)\eps_n)\right)^2 \dd x \\
	&\le \frac{(\sup_{x\in[0,x_0]}|f_0'(x)|)^2}{4f_0(\th_0)}\sum_{i=1}^{N'} \int_{(i-1)\eps_n}^{i\eps_n}(x-(i-1)\eps_n)^2\dd x\\
	&= \frac{(\sup_{x\in[0,x_0]} |f_0'(x)|)^2}{12f_0(\th_0)}\sum_{i=1}^{N'}\eps_n^3\lesssim \eps_n^2
	\end{align*}
	where $\zeta_i\in ((i-1)\eps_n,i\eps_n)$.
\end{proof}

By lemmas \ref{lem:mix} and \ref{lem:mix2} we have.
\begin{cor}
	Assume $f_0$ satisfies assumption \ref{ass:f_0}. There exists a discrete probability measure $\tilde{P}$, supported on $\{i\eps_n,\, 1\le i \le N'\} \cup \{y_j,\, 1\le j \le \tilde{N}\}$, with $\min_{1\le j\le \tilde{N}} y_j \ge x_0$, $N' = \lfloor x_0/\eps_n\rfloor$ and $\tilde{N} \lesssim 1/\eps_n$ such that
	\[  \int_0^\infty \left( \sqrt{f_0(x)} - \sqrt{f_{\tilde{P}}(x)} \right)^2 \dd x \lesssim \eps_n^2. \]
	Moreover, the sequence $\{y_i\}$ can be taken such that $|y_i - y_j| \geq 2\eps_n^2$ for all $i, j \le \tilde{N}$. 	
\end{cor}

For easy reference, we redefine the weights $\tilde{p}_j$ of the measure $\tilde{P}$ from this corollary so that we can write $\tilde{P}=\sum_{j=1}^{N'} \tilde{p}_j \delta_{j\eps_N} + \sum_{j=1}^{\tilde{N}} \tilde{p}_{N'+j} \delta_{y_j}$.

Next, we use the support points and masses of the constructed measure $\tilde{P}$. To this end, define
\begin{align*}
U_i&=(i\eps_n,(i+1)\eps_n] \quad \text{for} \quad i=1,\dots,N'\\
U_{N'+i}&=[\th_0\vee(y_i-\eps_n^2),y_i+\epsilon_n^2] \quad \text{for} \quad i=1,\dots,\tilde{N}\\
U_0&=[0,\infty)\cap(\cup_{i=1}^{\tilde{N}+N'} U_i)^c,
\end{align*}
such that $U_0, U_1,\ldots, U_{N'+\tilde{N}}$ is a partition of $[0,\infty)$. Now define
the following set of decreasing densities
\[\mathcal{N}_n=\{f_{P'}\,:\,  P'([0,\infty))=1,\, |P'(U_i)-\tilde{p}_i| \le\epsilon_n^2/\tilde{N},\, 1\le i \le \tilde{N}+N'\}\]
To prove that  $\mathcal{N}_n$ is a subset of $\mathcal{S}_n$ a key property is that the  measure $\tilde{P}$ is constructed such that  $\int_0^\infty \left(\sqrt{f_0}-\sqrt{f_{\tilde{P}}}\right)^2 \lesssim \eps_n^2$ (see the proof of lemma 8 in \cite{Salomond}). Moreover, the prior mass of $\scr{N}_n$ is tractable because  $U_0, U_1,\ldots, U_{N'+\tilde{N}}$  is a partition of $[0,\infty)$.

\begin{rem}
	If the set $\scr{N}_n$ is defined with the masses $p_1,\ldots, p_N$ from lemma \ref{lem:mix} (as is done in \cite{Salomond}, then the resulting sets $\{U_i\}$ do  not form a partition. This results in intractable expressions for $\Pi(\scr{N}_n)$. For that reason, we defined another discrete measure $\tilde{P}$ such that the support points are $2\eps_n^2$ separated thereby fixing the issue.
	
\end{rem}

The arguments for lower bounding $\Pi(\scr{N}_n)$ can now be finished as outlined in \cite{Salomond}.
Without loss of generality, for $n$ sufficiently large we can assume $\alpha G_0(U_i)<1$, for $i=0,1,\dots,N'+\tilde{N}$. Similar to Lemma 6.1 in \cit{Ghosal}, we have
\begin{align*}
\Pi(\mathcal{N}_n)&\ge Dir(P'(U_i)\in [\tilde{p}_i\pm\epsilon_n^2/\tilde{N}],\,  1\le i\le \tilde{N}+N')\\
&\ge \Gamma(\alpha)\prod_{i=1}^{N'+\tilde{N}}\frac{1}{\Gamma(\alpha G_0(U_i))}\int_{0\wedge(\tilde{p}_i-\epsilon_n^2/\tilde{N})}^{\tilde{p}_i+\epsilon_n^2/\tilde{N}}x_i^{\alpha G_0(U_i)-1} \dd x_i.
\end{align*}
Here we use $(P'(U_0))^{\alpha G_0(U_0)-1}\ge 1$. As $x_i^{\alpha G_0(U_i)-1}\ge 1$ we have $$\int_{0\wedge(p_i-\epsilon_n^2/\tilde{N})}^{p_i+\epsilon_n^2/\tilde{N}}x_i^{\alpha G_0(U_i)-1} \dd x_i\ge 2\epsilon_n^2\tilde{N}^{-1}.$$
Substituting this bound into the lower bound on $\Pi(\scr{N}_n)$, combined with the inequalities $\beta\Gamma(\beta)=\Gamma(\beta+1)\le 1$ for $0<\beta\le 1$ and $\tilde{N}\lesssim\epsilon_n^{-1}$, we obtain \[ \Pi(\mathcal{N}_n)\gtrsim \epsilon_n^{3(N'+\tilde{N})}\prod_{i=1}^{N'+\tilde{N}} G_0(U_i)=\exp\left(3(N'+\tilde{N})\log\epsilon_n+\sum_{i=1}^{N'+\tilde{N}}\log G_0(U_i)\right).\]
When $N'<i\le N'+\tilde{N}$ it is trivial that $G_0(U_i)\gtrsim\epsilon_n^2$ and therefore
$$\sum_{i=N'+1}^{N'+\tilde{N}}\log G_0(U_i)\gtrsim \tilde{N}\log\epsilon_n.$$

For bounding $G_0(U_i)$ when $i\le N'$, we use the property of $g_0$ in (\ref{eq:priormass}): $g_0(\th)\ge \underline{k} e^{-\underline{a}/\th}$. In this case we have
\[G_0(U_i)\ge\underline{k}\int_{U_i}e^{-\underline{a}/\th}d\th \ge\underline{k}\epsilon_n\exp(-\underline{a}/(i\epsilon_n)).
\]
Implying
\[\sum_{i=1}^{N'}\log G_0(U_i)\ge N'\log(\underline{k}\epsilon_n)-\underline{a}\epsilon_n^{-1}\sum_{i=1}^{N'} i^{-1}.\]
Since $\sum_{i=1}^{N'} i^{-1}\asymp \log (N')\asymp \log\epsilon_n^{-1}$, we therefore have
\begin{equation} \label{eq:track-low-mass}\sum_{i=1}^{N'}\log G_0(U_i)\gtrsim \epsilon_n^{-1}\log\epsilon_n.\end{equation}
Therefore, we obtain
$$ \Pi(\scr{S}_n) \ge \Pi(\mathcal{N}_n)\gtrsim e^{C_1\epsilon_n^{-1}\log \epsilon_n}\gtrsim e^{-C_1n\epsilon_n^2}$$
for some $C_1>0$. This is exactly as is required.

\section{Some details on the simulation in section \ref{sec:simul}}\label{sec:updating-theta}
In this section we provide some computational details for updating the $\th$-values in the MCMC-sampler.  Given the initialisation of $(X,Z,\Th)$, we numerically evaluate $\int \psi(x_i \mid \th) d G_0(\th)$ for $i=1,\dots,n$. If $g_0$ is not conjugate to the uniform distribution, we  use the random walk type Metropolis-Hastings method sampling from $f_{\Th_k\mid X,Z}$ using the normal distribution.
For update each $Z_i$, if $N_{Z_i,-i}=0$, we first remove $\Th_{Z_i}$. If we draw a new "cluster" for $Z_i$, $1+\vee(Z)$, then we also draw a new sample for $\Th_{Z_i}$ according to ($\ref{eq:fthk}$). In this case, the product $\prod_{j:z_j=k}\psi(x_j\mid\theta_k)$ only has one item, that is $f_{\Th\mid X,Z}(\th\mid x,z)\propto g_0(\th)\psi(x_i\mid\th)$. Sampling a value for $\th$ is done as follows:
\begin{enumerate}
	\item If the base density $g_0$ is as in \eqref{eq:expont}, then we use rejection sampling. To that end, if we set $Y=1/\Th$, then
	$$f_{Y\mid X,Z}(y\mid x,z)=\frac1{y^2}f_{\Th\mid X,Z}\left(\frac1{y}\mid x,z\right)=C \frac1{y}e^{-y-1/y}\mathbf{1}_{[0,1/x_i](y)},$$
	where $C$ is a constant such that $\int_0^\infty f_{Y\mid X,Z}(y\mid x,z)\,\mbox{d} y=1$.
	For reject sampling,  we choose the proposal density $g(y)$ to be uniform on $[0,1/x_i]$. Since $\frac1{y}e^{-y-1/y}\le 0.18$ for any $y>0$, an  upper bound for $\frac{f_Y(y)}{g(y)}$ is given by $M=\frac{0.18\cdot C}{x_i}$. Hence, we sample from $f_{Y\mid X,Z}$ as follows:
	\begin{enumerate}
		\item sample $y\sim g(y)$, $u\sim Unif(0,1)$;
		\item if \[ u\le \frac{f(y)}{Mg(y)}=\frac{C e^{-y-1/y}}{Myx_i}=\frac{e^{-y-1/y}}{0.18y},\] then accept and set $\th_{z_i}=1/y$; else return to step (a).
	\end{enumerate}
	\item If the base density $g_0$ is $Gamma(2,1)$, then
	\[f_{\Th\mid X,Z}(\th\mid x,z)=C e^{-\th}1_{[x_i,\infty)}(\th),\]
	where $C=1/\int_{x_i}^\infty e^{-\th}d\th=e^{x_i}$. Hence the cumulative distribution function  $F_\Theta$ satisfies $F_{\Th}(\th)=\int_{x_i}^\th Ce^{-t} dt=1-e^{x_i-\th}$, when $\th\ge x_i$.
	By the inverse cdf method, $\th$ can be sampled by first sampling $u\sim Unif(0,1)$ and next computing $x_i-\log(u)$. \end{enumerate}

\section{Results for the simulation experiment of Section \ref{sec:prior-influence} with sample size $n=1000$}
\label{sec:n=1000results_exp1}

The results with $n=1000$ are shown in figures \ref{fig:exp1000} and \ref{fig:halfnormal1000}. 

\begin{figure}[!htp]
	\begin{center}
		\includegraphics[scale=0.65]{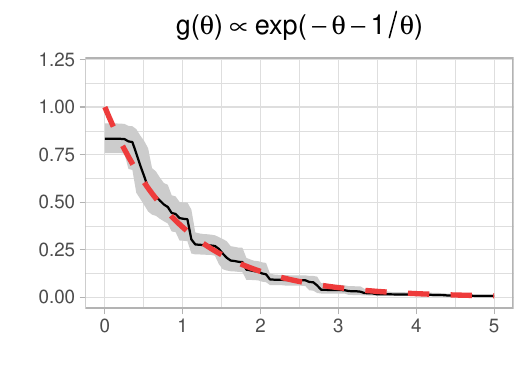}
		\includegraphics[scale=0.65]{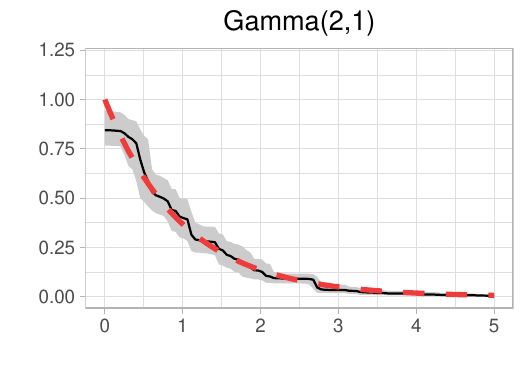}
		\includegraphics[scale=0.65]{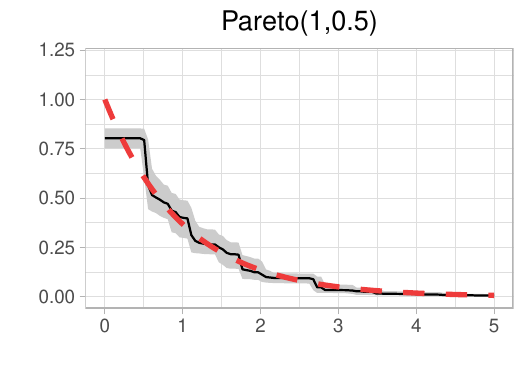}
		\includegraphics[scale=0.65]{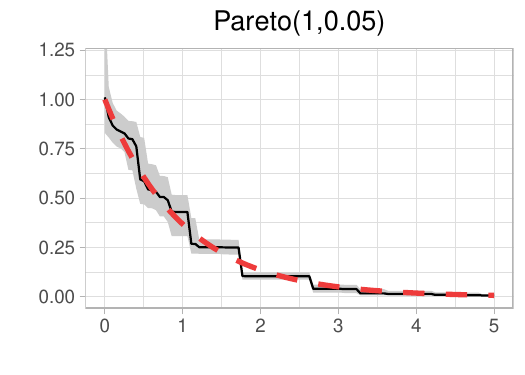}
		\includegraphics[scale=0.65]{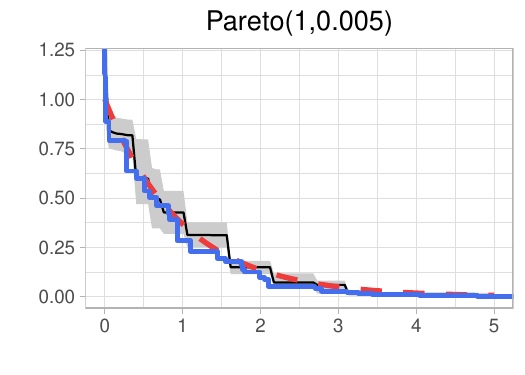}
		\includegraphics[scale=0.65]{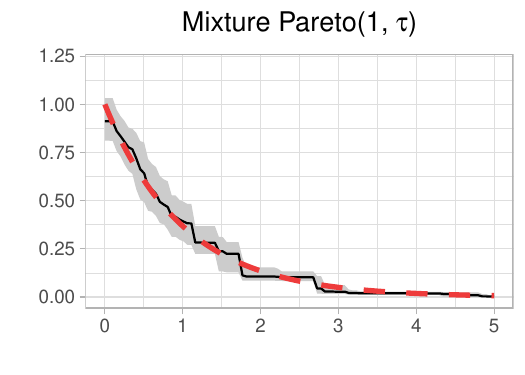}
	\end{center}
	\caption{Same experiment as in Figure \ref{fig:exp100}, this time with a sample of size $1000$ from the standard Exponential distribution. \label{fig:exp1000}}\end{figure}

\begin{figure}[!htp]
	\begin{center}
		\includegraphics[scale=0.65]{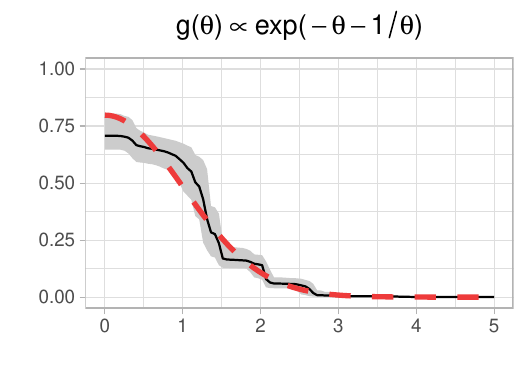}
		\includegraphics[scale=0.65]{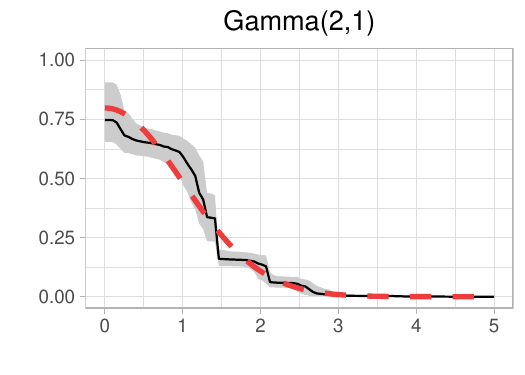}
		\includegraphics[scale=0.65]{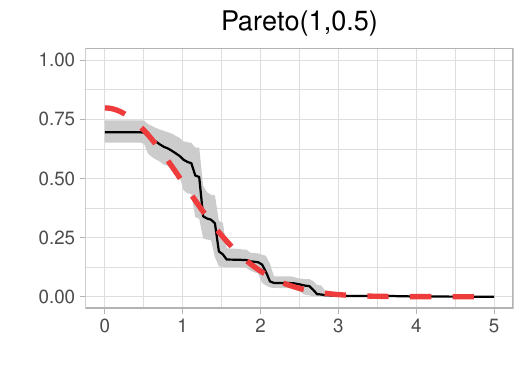}
		\includegraphics[scale=0.65]{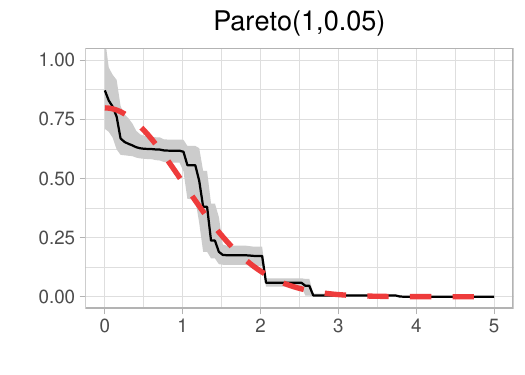}
		\includegraphics[scale=0.65]{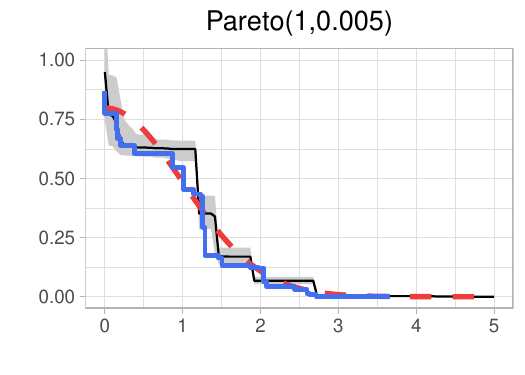}
		\includegraphics[scale=0.65]{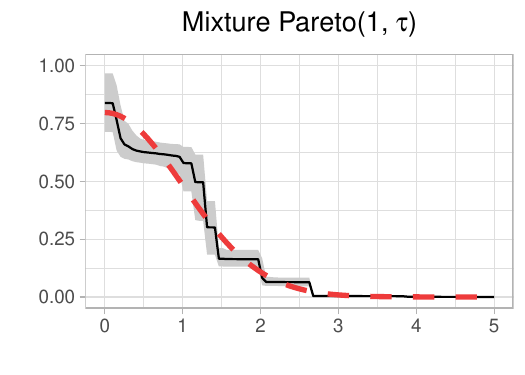}
	\end{center}
	\caption{Same experiment as in Figure \ref{fig:halfnormal100}, this time with a sample of size $1000$ from the halfNormal distribution. \label{fig:halfnormal1000}}
\end{figure}

%%%%%%%%%%%%%%%%%%%%%%%%%%%%

%\newpage
%\setcounter{page}{1}

\end{document}